\documentclass[11pt]{amsart}

\usepackage{fullpage}

\usepackage{amsmath, amscd, amsthm, amssymb}

\usepackage{hyperref}

\def\C{{\mathbf C}}
\def\R{{\mathbf R}}
\def\Z{{\mathbf Z}}
\def\Q{{\mathbf Q}}
\def\A{{\mathbf A}}

\newtheorem{theorem}{Theorem}[subsection]

\newtheorem{lemma}[theorem]{Lemma}

\newtheorem{proposition}[theorem]{Proposition}

\newtheorem{corollary}[theorem]{Corollary}

\theoremstyle{definition}

\theoremstyle{remark}
\newtheorem{remark}[theorem]{Remark}

\newcommand{\mm}[4]{\left(\begin{smallmatrix} #1 & #2\\ #3 & #4\end{smallmatrix}\right)}

\DeclareMathOperator{\SO}{SO}

\DeclareMathOperator{\GSp}{GSp}

\DeclareMathOperator{\SU}{SU}

\DeclareMathOperator{\GL}{GL}

\DeclareMathOperator{\GE}{GE}

\DeclareMathOperator{\diag}{diag}

\def\g{{\mathfrak g}}
\def\h{{\mathfrak h}}
\def\k{{\mathfrak k}}
\def\p{{\mathfrak p}}

\def\m{{\mathfrak m}}

\def\sl{{\mathfrak {sl}}}

\def\Wh{{\mathcal W}}
\def\Vm{{\mathbb{V}}}

\allowdisplaybreaks

\begin{document}
\title{The minimal modular form on quaternionic $E_8$}
\author{Aaron Pollack}
\address{Department of Mathematics\\ Duke University\\ Durham, NC USA}
\email{apollack@math.duke.edu}

\begin{abstract} Suppose that $G$ is a simple reductive group over $\Q$, with an exceptional Dynkin type, and with $G(\R)$ quaternionic (in the sense of Gross-Wallach).  In a previous paper, we gave an explicit form of the Fourier expansion of modular forms on $G$ along the unipotent radical of the Heisenberg parabolic.  In this paper, we give the Fourier expansion of the minimal modular form $\theta_{Gan}$ on quaternionic $E_8$, and some applications.  The $Sym^{8}(V_2)$-valued automorphic function $\theta_{Gan}$ is a weight four, level one modular form on $E_8$, which has been studied by Gan.  The applications we give are the construction of special modular forms on quaternionic $E_7, E_6$ and $G_2$.  We also discuss a family of degenerate Heisenberg Eisenstein series on the groups $G$, which may be thought of as an analogue to the quaternionic exceptional groups of the holomorphic Siegel Eisenstein series on the groups $\GSp_{2n}$.\end{abstract}
\maketitle


\setcounter{tocdepth}{1}
\tableofcontents
\section{Introduction} This paper is a sequel to the paper \cite{pollackQDS}.  In \cite{pollackQDS}, we studied ``modular forms'' on the quaternionic exceptional groups, following beautiful work of Gan-Gross-Savin \cite{ganGrossSavin} and Wallach \cite{wallach}.  We proved that these modular forms possess a refined Fourier expansion, similar to the Siegel modular forms on the symplectic groups $\GSp_{2n}$.  This is, in a sense, a purely Archimedean result: The representation theory at the infinite place on the quaternionic exceptional groups forces the modular forms to have a robust theory of the Fourier expansion.

Suppose that $G_{/\Q}$ is a quaternionic exceptional group of adjoint type.  The maximal compact subgroup $K_{\infty}\subseteq G(\R)$ is $(\SU(2) \times L)/\mu_2$ for a certain group $L$.  Denote by $\Vm_n = Sym^{2n}(V_2) \boxtimes \mathbf{1}$ the representation of $K_{\infty} = (\SU_2 \times L)/\mu_2$ that is the $(2n)^{th}$ symmetric power of the defining representation $V_2$ of $\SU(2)$ and the trivial representation of $L$.  Recall from \cite{pollackQDS} that if $n \geq 1$ is an integer, a modular form on $G$ of weight $n$ is a smooth, moderate growth function $F: G(\Q)\backslash G(\A) \rightarrow \Vm_n^\vee$ that satisfies
\begin{enumerate}
\item $F(gk) = k^{-1} \cdot F(g)$ for all $k \in K_{\infty} \subseteq G(\R)$ and
\item $\mathcal{D}_n F = 0$.\end{enumerate}
Here $\mathcal{D}_n$ is a certain first-order differential operator, closely-related to the so-called Schmid operator for the quaternionic discrete series representations on $G(\R)$.  It is that $F$ is annihilated by $\mathcal{D}_n$ that is the crucial piece of the definition of modular forms.

There is a weight four, level one modular form on quaternionic $E_8$ that is associated to the automorphic minimal representation, studied by Gan \cite{ganATM,ganSW}, which we denote by $\theta_{Gan}$.  The automorphic minimal representation is spherical at every finite place, but is not spherical at infinity; at the Archimedean place, it has minimal $K_{\infty}$-type $\Vm_4$.  If $v \in \Vm_4$, then pairing $\theta_{Gan}$ with $v$ gives the vector in the minimal representation that is $v$ at the Archimedean place and spherical at all the finite places.  Our main result is the complete and explicit Fourier expansion of this modular form.  See Theorem \ref{thm:thetaE8}.  Using $\theta_{Gan}$, we construct special modular forms on $E_7, E_6$ and $G_2$; see Corollary \ref{cor1} and Corollary \ref{cor2}.  Moreover, we study a family of absolutely convergent Eisenstein series on the quaternionic exceptional groups, and prove that all their nontrivial Fourier coefficients are Euler products; see Theorem \ref{thm:FEconvEis}.

To set up the statements of these results, let us recall from \cite{pollackQDS} the shape of the Fourier expansion of modular forms on the quaternionic exceptional groups.  Thus suppose $G = G_J$ is a quaternionic exceptional group of adjoint type, associated to a cubic norm structure $J$ over $\Q$ with positive definite trace form.  Then $G$ has a rational Heisenberg parabolic $P_J = H_J N_J$, with Levi subgroup $H_J$ and unipotent radical $N_J$.  The group $N_J \supseteq N_0$ is a two-step unipotent group, with center $N_0 = [N,N]$  and abelianization $N/N_0 \simeq W_J$.  Here $W_J = \Q \oplus J \oplus J^\vee \oplus \Q$ is Freudenthal's defining representation of the group $H_J$ (see \cite{pollackQDS} and \cite{pollackLL}).  

Suppose $F$ is a modular form of weight $n$ for $G$.  Denote by $F_0$ the constant term of $F$ along $N_0$, i.e., \[F_0(g) = \int_{N_0(\Q)\backslash N_0(\A)}{F(ng)\,dn}.\]
A simple argument using the left-invariance of $F$ under $G(\Q)$ proves that $F_0$ determines $F$ for the groups studied in \cite{pollackQDS}.  The Fourier expansion of $F_0$ is then given as follows: For $x \in (N/N_0)(\R) \simeq W_J(\R)$ and $g \in H_J(\R)$,
\[F_0(x g) = F_{00}(g) + \sum_{\omega \in W_J(\Q), \omega \geq 0}{a_F(\omega) e^{2 \pi i \langle \omega, x \rangle} \Wh_{2\pi\omega}(g)}\]
where here the notation is as follows:
\begin{itemize}
\item $F_{00}$ denotes the constant term of $F$ along $N$;
\item $a_F(\omega)$ is the Fourier coefficient associated to $\omega$;
\item $\langle \cdot,\cdot \rangle$ is Freudenthal's symplectic form on $W_J$;
\item $\Wh_{\omega}: H_J(\R) \rightarrow \Vm_n^\vee$ is a special function on $H_J(\R)$ defined in terms of the $K$-Bessel functions $K_v(\cdot)$ for $v \in \{-n,-n+1,\ldots,n-1,n\}$ and the element $\omega \in W_J(\R)$.
\end{itemize}

Denote by $x,y$ the fixed basis of $V_2$ from \cite{pollackQDS} so that $\{x^{2n}, x^{2n-1}y, \ldots, xy^{2n-1}, y^{2n}\}$ is a basis of $\Vm_n$.  The basis elements $x^{n+v}y^{n-v}$ of $\Vm_n$ are essentially characterized by the fact that $k \cdot x^{n+v}y^{n-v} = j(k,i)^v x^{n+v}y^{n-v}$ for $k \in K_H^1$ a certain compact subgroup of $H_J(\R)$ and $j$ the factor of automorphy on $H_J$ specified in \emph{loc cit}; see \cite[section 9]{pollackQDS}.  Then
\[\mathcal{W}_{\omega}(g) = \sum_{-n \leq v \leq n}{\mathcal{W}_{\omega}^v(g) \frac{x^{n+v}y^{n-v}}{(n+v)!(n-v)!}}\]
with 
\[\mathcal{W}_\omega^v(g) = \nu(g)^n |\nu(g)| \left(\frac{|\langle \omega, g r_0(i)\rangle|}{\langle \omega, g r_0(i)\rangle}\right)^v K_v(|\langle \omega, gr_0(i)\rangle|)\]
and $r_0(i) = (1,-i,-1,i) \in W_J \otimes \C$. Here $\nu: H_J \rightarrow \GL_1$ is the similitude character of $H_J$.  Moreover, the constant term
\[F_{00}(g) = \nu(g)^{n}|\nu(g)|\left(\Phi(g) \frac{x^{2n}}{(2n)!} + \beta \frac{x^n y^n}{n! n!} + \Phi'(g) \frac{y^{2n}}{(2n)!}\right)\]
for some holomorphic modular form $\Phi$ of weight $n$ on $H_J$ and $\Phi'(g) = \Phi(gw_0)$ for a specific element $w_0 \in H_J$ that exchanges the upper and lower half-spaces $\mathcal{H}_J^{\pm}$.

With this result recalled, let us now state the Fourier expansion of $\theta_{Gan}$.  Denote by $\Theta_0$ Coxeter's integral octonions \cite[(5.1)]{elkiesGrossIMRN} and $J_0 = H_3(\Theta_0)$ the associated integral lattice in the exceptional cubic norm structure $J$.  The Freudenthal space $W_J$ has a natural integral lattice $W_J(\Z) = \Z \oplus J_0 \oplus J_0^\vee \oplus \Z$.  For $\omega \in W_J(\Z)$, define $\Delta(\omega)$ to be the largest positive integer so that $\omega \in \Delta(\omega)W_J(\Z)$.  For $T \in J_0$, define $\Delta(T)$ analogously.

Recall Kim's modular form $H_{Kim}(Z)$ \cite{kimE7} on the exceptional tube domain, which has Fourier expansion
\[H_{Kim}(Z) = \frac{1}{240} + \sum_{T \in J_0, T \geq 0 \text{  rank one}}{\sigma_3(\Delta(T)) q^{T}}.\]
Denote by $\Phi_{Kim}$ the automorphic form on $H_J=GE_7$ so that $j(g,i)^4 \Phi_{Kim}(g)$ descends to $\mathcal{H}_J^{\pm}$, is holomorphic on $\mathcal{H}_J^{+}$, antiholomorphic on $\mathcal{H}_J^{-}$, and on $\mathcal{H}_J^{+}$ one has $H_{Kim}(Z) = j(g,i)^4 \Phi_{Kim}(g)$ if $Z = g \cdot i$.
\begin{theorem}\label{thm:thetaE8} Let the notations be as above.  Then
\[ \theta_{Gan,0}(xg) = \theta_{Gan,00}(g) + \frac{1}{12} \sum_{\omega \in W_J(\Z) \text{ rank one}}{\sigma_4(\Delta(\omega)) e^{2\pi i\langle \omega, x\rangle}\Wh_{2\pi \omega}(g)}\]
with
\[\theta_{Gan,00}(g) = |\nu(g)|^{5}\left( \frac{\zeta(5)}{\pi^4 2^4} \frac{x^4 y^4}{4!4!} + \frac{2}{3}\left(\Phi_{Kim} \frac{x^8}{8!} + \Phi_{Kim}' \frac{y^{8}}{8!}\right)\right).\]
\end{theorem}

There is a degenerate Heisenberg Eisenstein series on each of the quaternionic exceptional groups, which we write as $E(g,s;n)$.  This is a function $E(g,s;n)=E^{G}(g,s;n): G(\Q)\backslash G(\A) \rightarrow \Vm_n^\vee$ (depending on $s$) satisfying $E(gk,s;n) = k^{-1} E(g,s;n)$ for all $k \in K_\infty$.  When $G$ is quaternionic $E_8$ and $n=4$, it turns out that the Eisenstein series $E(g,s;4)$ is regular at $s=5$, and $\theta_{Gan}$ is defined \cite{ganATM} (up to a nonzero scalar multiple) as the value of this Eisenstein series at this point.  

Combining the archimedian results of \cite{pollackQDS} with work of Gan \cite{ganATM,ganSW,ganRegularized}, Kim \cite{kimE7}, Gross-Wallach \cite{grossWallach2}, Kazhdan-Polishchuck \cite{kazhdanPolishchuk} and Gan-Savin \cite{ganSavin}, most of Theorem \ref{thm:thetaE8} was known.  See the discussion in subsection \ref{subsec:atm}.  What is left is to pin down a couple constants.  We do this by analyzing the Fourier expansion $E^{G_2}(g,s=5;4)$ of the weight $4$ Eisenstein series on $G_2$, and applying the Siegel-Weil theorem of \cite{ganSW} which relates $\theta_{Gan}$ to this Eisenstein series on $G_2$.  

The Eisenstein series $E^{G_2}(g,s=5;4)$ is easier to compute with in relation to $E^{E_8}(g,s=5;4)$ (which defines $\theta_{Gan}$) because $s=5$ is in the range of absolute convergence of $E^{G_2}(g,s;4)$ but not of $E^{E_8}(g,s;4)$.  In fact, we study absolutely convergent degenerate Heisenberg Eisenstein series $E^{G}(g,s;n)$ on the quaternionic exceptional groups $G$ in general, which is our second main result.  More precisely, if $n$ is even and the special point $s=n+1$ is in the range of absolute convergence, the Eisenstein series $E(g,s=n+1;n)$ is a modular form on $G$ of weight $n$.   We prove that at such an $n$, all of the nontrivial Fourier coefficients of $E(g,s=n+1;n)$ are Euler products.  This is the analogue to the exceptional groups of the corresponding classical fact about holomorphic Siegel Eisenstein series on symplectic groups \cite{siegel}; we defer a precise statement of this result to section \ref{sec:FEeis}.  The proof is an easy consequence of a weak form of the main result of \cite{pollackQDS}: The Fourier expansion of $E(g,s;n)$ has many terms, some of which are Euler products and some of which are not.  However, applying \cite{pollackQDS}, one can deduce that all of the terms that are not Euler products vanish at $s=n+1$ for purely Archimedean reasons.

We also give a few applications of Theorem \ref{thm:thetaE8}, to modular forms on $E_7$, $E_6$ and $G_2$.  Namely, one can pull back the minimal modular form $\theta_{Gan}$ on $E_{8}$ to the simply-connected quaternionic $E_7$ and $E_6$.  Denote these pull-backs by $\theta^{(2)}_{E_7}$ and $\theta^{(4)}_{E_6}$, respectively.  These pull-backs give interesting singular and distinguished modular forms on $E_7^{sc}$ and $E_6^{sc}$.  The modular form $\theta^{(2)}_{E_7}$ is singular in that it has no rank three or rank four Fourier coefficients, but it does have nonzero rank two Fourier coefficients.  The modular form $\theta^{(4)}_{E_6}$ is not singular--it has nonzero rank four Fourier coefficients.  However, it is distinguished in that it has only one orbit of nonzero rank four Fourier coefficients.

\begin{corollary}\label{cor1} The automorphic functions $\theta^{(2)}_{E_7}$ and $\theta^{(4)}_{E_6}$ define nonzero modular forms on $E_7^{sc}$ and $E_6^{sc}$ of weight $4$.  Moreover
\begin{enumerate}
\item The modular form $\theta^{(2)}_{E_7}$ has nonzero rank two Fourier coefficients, but all of its rank three and rank four Fourier coefficients are $0$;
\item The modular form $\theta^{(4)}_{E_6}$ is distinguished: it has only one orbit of nonzero rank four Fourier coefficients. \end{enumerate}
\end{corollary}
The distinguished nature of these Fourier expansions is a more-or-less immediate consequence of the results of \cite[section 7 and 8]{pollackLL}.

Note that theorem \ref{thm:thetaE8} says that (a scalar multiple of) $\theta_{Gan}$ just fails to have integral Fourier coefficients.  All the rank one Fourier coefficients are integers, and all the Fourier coefficients of $\Phi_{Kim}$ are integers.  Thus, it is reasonable to ask for a nonzero modular form on an exceptional group for which \emph{all} of its Fourier coefficients are integers.  Our final application of Theorem \ref{thm:thetaE8} is to produce such a modular form on $G_2$.  

Following \cite{elkiesGrossIMRN,ganGross,ganGrossSavin}, there is a unique (up to scaling) automorphic function $\epsilon$ on a certain anisotropic form of $F_4$, that is right invariant under $F_4(\widehat{\Z}) F_4(\R)$ and orthogonal to the constant functions.  Denote by $F_{\Delta}$ the theta lift of $\epsilon$ to $G_2$ via $\theta_{Gan}$.  The modular form $F_{\Delta}$ is discussed in \cite{ganGrossSavin} and \cite{ganGross}.  The following is an essentially immediate corollary of Theorem \ref{thm:thetaE8} and results of \emph{loc cit}.
\begin{corollary}\label{cor2} The modular form $F_{\Delta}$ has rational Fourier coefficients with bounded denominators.  Its constant term is proportional to Ramanujan's function $\Delta$. \end{corollary}
 
\subsection{Acknowledgments} We thank Wee Teck Gan and Benedict Gross for their encouragement and helpful comments.

\subsection{Notation} Throughout the paper, the notation is as in \cite{pollackQDS}.  In particular, $F$ denotes a field of characteristic $0$, $J$ denotes a cubic norm structure over $F$, and $\g_J$ or $\g(J)$ the Lie algebra associated to $J$ in \cite[Section 4]{pollackQDS}. The field $F$ will frequently be $\Q$ or $\R$.  We will assume that $J$ is either $F$, or $H_3(C)$ with $C$ a composition algebra over $F$.  Thus $\g(J)$ is of type $G_2, F_4, E_6, E_7,$ or $E_8$. 

The letter $K$ or $K_\infty$ denotes the maximal compact subgroup of $G_J(\R)$ defined in \emph{loc cit}, where $G=G_J$ is the adjoint group associated to the Lie algebra $\g(J)$.  We will sometime refer to elements of $\g(J)$ is the notation of the $\Z/3$-model, and other times refer to elements of this Lie algebra in the $\Z/2$-model.  Again, see \cite[section 4]{pollackQDS}.  We write $\Vm_n$ for the representation of $K = (\SU(2) \times L)/\mu_2$ on $Sym^{2n}(V_2) \boxtimes \mathbf{1}$.  \emph{Modular forms} on $G_J$ are by definition certain functions $F: G_J(\Q)\backslash G_J(\A) \rightarrow \Vm_n^\vee$ satisfying $F(gk) = k^{-1} \cdot F(g)$ for all $g \in G_J(\A)$ and $k \in K$, which are annihilated by a first-order differential operator $\mathcal{D}_n$.

One defines $P_J = H_J N_J$ (or $P = HN$, if $J$ is fixed) to be the Heisenberg parabolic of $G_J$, which by definition is the stabilizer of the line $F E_{13}$ in $\g(J)$. We write $N_0 = [N,N]$, which is also the center of $N$.  The letter $\nu$ denotes the similitude character of $P$; one has $\nu: P \rightarrow \GL_1$ given by $p \cdot E_{13} = \nu(p) E_{13}$.

We write $\h(J)$ for the Lie algebra of the Freudenthal group $H_J$ and $\m(J)$ for the Lie algebra of the group that preserves the cubic norm on $J$ up to similitude.  Then (see \cite[section 4]{pollackQDS} for our normalizations)
\begin{align*} \g(J) &= \sl_2 \oplus \h(J)^0 \oplus V_2 \otimes W_J \\ & \simeq \sl_3 \oplus \m(J)^0 \oplus V_3 \otimes J \oplus (V_3 \otimes J)^\vee.\end{align*}
For $w_1, w_2 \in W_J$ we denote by $\Phi_{w_1,w_2}$ the element of $\h(J)^0$ specified in \cite[section 3.4.2]{pollackQDS}.

Finally, if $z \in \C$ and $j \geq 0$ an integer, $(z)_j = z(z+1)(z+2) \cdots (z+j-1) = \frac{\Gamma(z+j)}{\Gamma(z)}$ is the Pochhammer symbol.

\section{Statement of results, and applications}\label{sec:statements} In this section we state our main results more precisely, and give the proofs of Corollary \ref{cor1} and Corollary \ref{cor2}.  We begin by defining the degenerate Heisenberg Eisenstein series on the quaternionic groups $G_J$, as these Eisenstein series are central to everything that follows.  We then review what was known about the automorphic form $\theta_{Gan}$.  Finally, we restate Corollary \ref{cor1} and \ref{cor2} and give the proofs of these results.

\subsection{The degenerate Heisenberg Eisenstein series} In this subsection, we define the degenerate Heisenberg Eisenstein series $E_J(g,s;n)$ on the quaternionic exceptional groups $G_J$. By definition, such an Eisenstein series is associated to a section $f(g,s) \in Ind_{P(\A)}^{G_J(\A)}(|\nu|^{s})$.  More precisely, we use the final parameter $n$ in $E(g,s;n)$ to indicate that $E(g,s;n)$ is $\Vm_n^\vee$-valued and satisfies $E(gk,s;n) = k^{-1} \cdot E(g,s;n)$.  Throughout, \emph{we will assume that $n \geq 0$ is even}.

We now construct such an Eisenstein series explicitly; this makes it easier to do computations. Suppose $\Phi_f$ is a Schwartz-Bruhat function on $\g_J(\A_f)$.  We will define a $\Vm_n$-valued Schwartz function $\Phi_{\infty,n}$ on $\g_J(\R)$ satisfying $\Phi_{\infty,n}(k v) = k \cdot \Phi_{\infty,n}(v)$ momentarily.  With this definition, we set $\Phi = \Phi_f \otimes \Phi_{\infty,n}$ and then
\[f(g,\Phi,s) = \int_{\GL_1(\A)}{|t|^{s}\Phi(t g^{-1} E_{13})\,dt}.\]
It is clear that $f(g,\Phi,s)$ is a section in the induced representation $Ind_{P_J}^{G_J}(|\nu|^{s})$, and we set
\[E(g,\Phi,s) = \sum_{\gamma \in P_J(\Q)\backslash G_J(\Q)}{f(\gamma g, \Phi,s)}.\]
We will be interested in this Eisenstein series at the special value $s=n+1$.  When $n > \dim(W_J)/2 = \dim J + 1$, the Eisenstein series converges absolutely at $s=n+1$ and defines a modular form there; see the remarks after Corollary 1.2.4 in \cite{pollackQDS}.

The special archimedean function $\Phi_{\infty,n}$ is defined as follows.  Denote by $\k_2$ the $\mathfrak{su}_2$ part of $\k$, the Lie algebra of the maximal compact $K$.  Denote by $pr: \g(J) \rightarrow \k_2$ the $K$-equivariant projection. For $n \geq 0$, define $\Phi_{\infty, n}(v) = pr(v)^n e^{-\pi ||v||^2}$.  Here $||v||^2 = B_{\g}(v,-\Theta(v))$, with $B_\g$ and $\Theta$ defined in \cite[section 4]{pollackLL}. It is clear that $\Phi_{\infty,n}(k v) = k\cdot \Phi_{\infty,n}(v)$.

\subsection{The minimal automorphic forms on quaternionic $E_8$}\label{subsec:atm}  In this subsection, we briefly discuss the automorphic miniminal representation on quaternionc $E_8$.  The reader should see \cite{ganATM} and \cite{ganSavin} and the references contained therein for more details.

For this subsection, let $J = H_3(\Theta)$ with $\Theta$ the octonion algebra over $\Q$ whose trace pairing is positive definite.  Then $G_J$ is the quaterionic $E_8$.  Suppose that $f_s \in Ind_{P_J(\A)}^{G_J(\A)}(|\nu|^{s})$ is a flat section, and $E_J(g,f_s)$ the associated Eisenstein series.  It is proved in \cite{ganATM} that for appropriate $f_s$, $E_J(g,f_s)$ has a simple pole at $s=24$.  Moreover, this pole can be achieved when $f_s$ is spherical at every finite place.  The automorphic minimal representation $\Pi$ is defined \cite{ganATM} to be the space of residues of the $E_J(g,f_s)$ at $s=24$.  By e.g. \cite{magaardSavin} and also \cite{ganSavin}, the space of such automorphic forms only have rank $1$ and rank $0$ Fourier coefficients along $N_J$; for instance, this follows by the analogous local fact for one finite place.

Denote by $E_J(g,s)$ the Eisenstein series associated to the flat section $f_J(g,s;n)$ which has the following properties:
\begin{enumerate}
\item $f_J(g,s;n)$ is valued in $\Vm_n^\vee \simeq \Vm_n$, and satisfies $f_J(gk,s;n) = k^{-1} f_J(g,s;n)$ for all $g \in G_J(\A)$ and $k \in K \subseteq G_J(\R)$;
\item $f_J$ is spherical at every finite place;
\item $f_J(1,s;n) = \frac{x^n y^n}{n!n!} \in \Vm_n$.\end{enumerate}
One defines $\theta_{Gan}$ to be a certain nonzero multiple of $Res_{s=24}(E_J(g,s;4))$.  It is proved in \cite{ganSavin} by a somewhat indirect method that $\theta_{Gan}$ is nonzero, i.e. that $E_J(g,s;4)$ does have a nontrivial pole at $s=24$.  Below we will give a direct proof of this.  More precisely, we shall use the following fact.
\begin{proposition}\label{prop:EisReg1} The Eisenstein series $E_J(g,s;4)$ is regular at $s=5$, and defines a modular form on $G_J$ of weight $4$.  Up to nonzero scalar multiples, 
\[E_J(g,s=5;4) = Res_{s=24} E_J(g,s;4) = \theta_{Gan}.\]
\end{proposition}
This proposition is essentially contained in \cite{ganATM}, \cite{ganSW}, \cite{ganSavin}, \cite{grossWallach2}.  As it is crucial to the main results of this paper, we spell out a direct proof of it in section \ref{sec:mmf}.

Now, because $\theta_{Gan}$ is a modular form on $G_J$, the results of \cite{pollackQDS} imply that its Fourier expansion takes the following shape.  Denote by $\Theta_0$ Coxeter's integral subring \cite[(5.1)]{elkiesGrossIMRN} of $\Theta$, by $J_0 = H_3(\Theta_0)$, and $W_J(\Z) = \Z \oplus J_0 \oplus J_0^\vee \oplus \Z \subseteq W_J(\Q)$.  Then for $x \in N_J(\R)$ and $m \in H_J(\R)$ one has
\[\theta_{Gan,0}(n(x)m) = \theta_{00}(m) + \sum_{\omega \in W_J(\Z)}{a_{\theta}(\omega)e^{2\pi i \langle x, \omega\rangle}\mathcal{W}_{2\pi\omega}(m)}\]
with the constant term $\theta_{00}(m)$ given by
\[\theta_{00}(m) = \beta_1 \Phi(m) x^8 + \beta_0 x^4 y^4 + \beta_1 \Phi'(m) y^{8}\]
for a holomorphic weight $4$ modular form $\Phi$ on $H_J = \GE_7$.  Because $\theta_{Gan}$ is minimal, $a_{\theta}(\omega)$ is nonzero only for $\omega$ rank one.  Denote by $\Delta(\omega)$ the largest positive integer so that $\omega \in \Delta(\omega) W_J(\Z)$.  By \cite{ganRegularized} and \cite{kazhdanPolishchuk}, we may scale $\theta_{Gan}$ so that
\[a_{\theta}(\omega) = \begin{cases} \sigma_{4}(\Delta(\omega)) &\mbox{if } \omega \text{ is rank one} \\ 0 &\mbox{if } \omega \text{ is rank two, three, or four}.\end{cases}\]

Moreover, from \cite{ganSW}, $\Phi$ is proportional to Kim's \cite{kimE7} level one, weight $4$ modular form on $H_J=\GE_7$.  Thus, applying the results of \cite{grossWallach2}, \cite{ganATM}, \cite{ganSW}, \cite{ganSavin}, \cite{kazhdanPolishchuk}, \cite{kimE7} and \cite{pollackQDS}, what is left is to pin down the constants $\beta_0$ and $\beta_1$.  This is precisely what Theorem \ref{thm:thetaE8} does.  We restate the result now.

\begin{theorem}\label{thm:thetaAgain} The Eisenstein series $E_J(g,s;4)$ is regular at $s=5$ and defines a modular form on $G_J=E_{8,4}$ of weight $4$ at this point. The Fourier coefficient corresponding to the rank one element $(0,0,0,1) \in W_J$ is nonzero. Denote by $\theta_{Gan}$ the scalar multiple of $E_J(g,s;4)$ for which this Fourier coefficient is equal to $\frac{1}{24}$.  Moreover, denote by $\Phi_{Kim}$ the spherical automorphic form on $H_J=GE_7$ so that $H_{Kim} = j(g,i)^4 \Phi_{Kim}$ descends to $\mathcal{H}_J^{\pm}$, is holomorphic on $\mathcal{H}_J^{+}$, antiholomorphic on $\mathcal{H_J}^{-}$, and on $\mathcal{H}_J^{+}$ has the Fourier expansion
\[H_{Kim} = \frac{1}{240} + \sum_{T \in J_0, T \geq 0 \text{  rank one}}{\sigma_3(\Delta(T)) q^{T}}.\]
Then one has
\[\theta_{Gan,0} = |\nu(g)|^{5}\left( \frac{\zeta(5)}{\pi^4 2^5} \frac{x^4 y^4}{4!4!} + \frac{1}{3}\left(\Phi_{Kim} \frac{x^8}{8!} + \Phi_{Kim}' \frac{y^{8}}{8!}\right)\right) + \frac{1}{24} \sum_{\omega \in W_J(\Z), \text{ rank one}}{\sigma_4(\Delta(\omega)) \Wh_{2\pi \omega}(g)}.\]
\end{theorem}

We will prove this theorem in section \ref{sec:mmf} after understanding the Fourier expansion of degenerate absolutely convergent Heisenberg Eisenstein series in section \ref{sec:FEeis}.  We now detail and prove the corollaries of Theorem \ref{thm:thetaAgain} that were mentioned in the introduction.

\subsection{The singular modular form} In this subsection, we consider the singular modular form $\theta^{(2)}_{E_7}$ on the simply-connected quaternionic $E_7$.  This modular form is defined as follows.  First, fix a quaternion algebra $B$ over $\Q$, which is ramified at the archimedean place.  Recall that the quaternionic Lie algebra $\mathfrak{e}_7$ is $\g(H_3(B))$, in the notation of \cite{pollackQDS}.  For ease of notation, we write $J_B = H_3(B)$ and $J_{\Theta} = H_3(\Theta)$.

Now, fix $\gamma \in \Q^\times$ not representing the identity coset in $\Q^{\times}/N(B^\times)$; in other words, $\gamma < 0$.  By the Cayley-Dickson construction, one can form an octonion algebra $\Theta$ out of $B$ and $\gamma$.  See \cite[section 8]{pollackLL}.  With such a $\gamma$, $\Theta$ is ramified at infinity.  The Cayley-Dickson construction induces an identification $J_B \oplus B^3 \simeq J_{\Theta}$, an embedding $\h(J_B) \rightarrow \h(J_\Theta)$, and then consequently an embedding $\g(J_B) \hookrightarrow \g(J_\Theta)$.

More precisely, denote by $W_6$ the defining representation of $\GSp_6$, and define an identification $W_{J_B} \oplus W_6 \otimes B \simeq W_{J_\Theta}$ as in \cite[section 8.1.2]{pollackLL}.  From \cite[Proposition 8.1.5]{pollackLL}, one gets a group $H_B'$ (this is the group $G(\gamma,C)$ in the notation of that proposition) together with maps $H_B' \rightarrow H_{J_B}$ and $H_B' \rightarrow H_{J_\Theta}$ where the first map induces an isomorphism of Lie algebras.  Consequently, one obtains a map $\h^0(J_B) \rightarrow \h^0(J_\Theta)$.  As $\g(J) = \sl_2 \oplus \h^{0}(J) \oplus V_2 \otimes W_J$, one obtains a specific embedding $\g(J_B) \rightarrow \g(J_\Theta)$.  

Denote by $A_B$ the connected component of the identity of the subgroup of $G_{J_\Theta}$ that preserves $\g(J_B)$.  We define a map $B^{n=1} \rightarrow A_B$ as follows.  First, define $B^{n=1} \rightarrow H_{J_\Theta}$ via its action on $W_{J_\Theta} \simeq W_{J_B} \oplus B^6$ as $s \cdot (w,v) = (w,vs^{-1})$ for $s \in B^{n=1}$, $w \in W_{H_3(B)}$ and $v \in B^6$.  Because the quadratic norm on $\Theta$ is $n_{\Theta}(x,y) = n_B(x) - \gamma n_B(y)$ for $x,y \in B$, it is easy to see directly that this action preserves the symplectic and quartic form on $W_{H_3(\Theta)}$.  Now because $H_{J_\Theta} \rightarrow G_{J_\Theta}$, this defines $B^{n=1} \rightarrow G_{J_\Theta}$.  Finally, it is clear by construction that this $B^{n=1}$ preserves $\g(J_B)$, and thus we obtain $B^{n=1}\rightarrow A_B$, as claimed.

Denote by $E_B$ the connected component of the identity of the centralizer of $B^{n=1}$ in $A_B$. 
\begin{lemma} The group $E_B$ is the simply connected quaternionic $E_7$.\end{lemma}
\begin{proof} Indeed, one verifies without difficulty that $\g(J_\Theta)^{B^1} = \g(J_B)$, and thus $E_B$ has the correct Lie algebra.  Moreover, the $\mu_2 \subseteq B^1$ centralizes $B^1$ in $A_B$, and this $\mu_2$ sits in $H_B'$, see \cite[Proposition 8.1.5]{pollackLL}.  Because $H_B'$ is connected and centralizes $B^{1}$, this proves that $\mu_2$ is in the center of $E_B$.  Thus $E_B$ is connected, has Lie algebra equal to $\g(J_B)$, and contains $\mu_2$ in its center, so $E_B \simeq E^{sc}_{7,4}$.\end{proof}

Denote by $\theta^{(2)}$ the automorphic function that is the pullback of $\theta_{Gan}$ to $E_B$ via the embedding $E_7^{sc} \simeq E_B \rightarrow G_{J_\Theta}$. Compare \cite{grossWallach1} and \cite{loke1,loke2}.  We have the following result, which is a restatement of Corollary \ref{cor1} (1).

\begin{proposition}\label{prop:theta2} The automorphic function $\theta^{(2)}$ is a modular form on $E_7^{sc}$ of weight $4$.  It has nonzero rank two Fourier coefficients, but all of its rank three and rank four Fourier coefficients are $0$. \end{proposition}
The definitions and results of \cite{pollackQDS} were made for adjoint groups, not simply connected ones.  However, it is easy to see that they carry over immediately for the simply connected $E_7$.  Indeed, because the $\mu_2$ that is the center of $E_7^{sc}$ acts trivially on $\Vm_n$, and because the map of real groups $E_7^{sc}(\R) \rightarrow E_7^{ad}(\R)$ is surjective \cite{thang}, the archimedean theory is identical for modular forms on the adjoint $E_7$ and modular forms on the simply connected $E_7$.
\begin{proof}[Proof of Proposition \ref{prop:theta2}] To see that $\theta^{(2)}$ is a modular form, one must only check the condition $\mathcal{D}_4 \theta^{(2)} = 0$.  One way to do this is simply observe that since $\theta_{Gan}$ satisfies the equations of \cite[Theorem 7.3.1 or Theorem 7.5.1]{pollackQDS}, so too does $\theta^{(2)}$.  One can also reason directly with the definition of $\mathcal{D}_4$ in terms of a basis of $\p(J_B)$ and $\p(J_\Theta)$, or apply results of \cite{grossWallach1,loke1,loke2}.

For the analysis of the Fourier coefficients, this is a direct consequence of \cite[Theorem 8.1.4]{pollackLL}.  Namely, if $x \in W_{J_B}$ is nonzero, and $a(x)$ denotes the $x$-Fourier coefficient of the modular form $\theta^{(2)}$, then 
\[a(x) = \sum_{u \in W_6(B)}{a_{\theta_{Gan}}(x + u)}.\]
(The sum has only finitely many nonzero terms.)  Because all the numbers $a_{\theta_{Gan}}(\omega)$ are non-negative, it is clear that $\theta^{(2)}$ is nonzero.  Finally, because $a_{\theta_{Gan}}(\omega)$ is only nonzero for $\omega$ rank one, all the rest of claims of the proposition follow immediately from \cite[Theorem 8.1.4]{pollackLL}.  This completes the proof.\end{proof}

\subsection{The distinguished modular form} Pulling back $\theta_{Gan}$ to the semisimple simply-connected quaternionic $E_6$, we obtain a modular form $\theta^{(4)}$.  In this subsection, we discuss the automorphic form $\theta^{(4)}$, and explain why it is distinguished.

Fix a quadratic imaginary extension $K$ of $\Q$.  Recall that the Lie algebra $\g(H_3(K))$ is the quaternionic Lie algebra of type $E_6$.  Using $H_3(K)$ and some additional data, the so-called second construction of Tits produces an exceptional cubic norm structure $J$.  We will use this construction to define the map $E_6^{sc}\rightarrow G_J\simeq E_{8,4}$ and analyze the Fourier coefficients of $\theta^{(4)}$.  

Thus, suppose $\lambda \in K^\times$, $S \in H_3(K)$ and that $\lambda \lambda^* = N(S)$.  In this subsection, we let $J_K$ denote $H_3(K)$ and $B$ denote $M_3(K)$.  Set $J = H_3(K) \oplus M_3(K) = J_K \oplus B$.  Then one can make $J$ into a cubic norm structure using $\lambda$ and $S$; see, e.g. \cite[section 7.1]{pollackLL}.  If $S$ is positive definite, then so is the trace pairing on $J$, and thus the Lie algebra $\g(J)$ is of type $E_8$ and quaternionic at infinity.  We will choose $\lambda =1$ and $S= 1_{3}$, so that $J \simeq H_3(\Theta)$ over $\Q$, but other choices of $\lambda, S$ should yield interesting results\footnote{At this point, we have only understood the automorphic form $\theta_{Gan}$ when $J = H_3(\Theta)$, because in the computations above and below we used that $\theta_{Gan}$ was spherical at every finite place.}.  

Now, via this construction of Tits, we obtain $J_K \rightarrow J$ and then $\h(J_K) \rightarrow \h(J)$ and then finally $\g(J_K)\rightarrow \g(J)$.  More precisely, from \cite[section 7.2]{pollackLL} there is an identification $W_{J_K} \oplus B^2 \simeq W_{J}$.  From \cite[Proposition 7.2.2]{pollackLL}, there is a group $H_K'$ (denoted $G$ in that proposition) that comes with maps $H_K' \rightarrow H_{J_K}$ and $H_K' \rightarrow H_J$, and it is easy to see that the first map induces an isomorphism of Lie algebras.  Consequently, as above, these constructions define an embedding $\g(J_K) \rightarrow \g(J)$.

Denote by $A_K$ the connected component of the identity of the subgroup of $G_J$ that preserves $\g(J_K)$.  We will construct explicitly the simply-connected quaternionic $E_6$ inside $A_K$, just as we did for $E_7$ in the previous subsection.  More precisely, consider the subgroup $\SU_3$ of $B^\times$ defined as the $g \in B$ with $\det(g) =1$ and $gSg^* = S$.  Let this $\SU_3$ act on $J \simeq J_K \oplus B$ as $g \cdot (X,\alpha) = (X,\alpha g^{-1})$ for $X \in J_K$, $\alpha \in B$, and $g \in \SU_3$.  It is clear from the formulas defining the second construction of Tits \cite[section 7.1]{pollackLL} that this action preserves the norm and pairing on $J$.  It acts on $W_{J} = W_{J_K} \oplus B^2$ as $g \cdot (w,\eta) = (w, \eta g^{-1})$, for $w \in W_{J_K}$ and $\eta \in B^2$.  Consequently, one obtains maps $\SU_3 \rightarrow H_J \rightarrow G_J$, and this $\SU_3$ lands in $A_K$ because its action on $\g(J)$ fixes $\g(J_K)$.

Denote by $E_K$ the connected component of the identity of the centralizer of this $\SU_3$ in $A_K$. 
\begin{lemma} The group $E_K$ is the simply-connected quaternionic $E_6$.\end{lemma}
\begin{proof} Indeed, one verifies quickly that $\g(J)^{\SU_3} = \g(J_K)$, and thus the Lie algebra of $E_K$ is $\g(J_K)$. Moreover, the $\mu_3$ that is the center of $\SU_3$ is identified with the diagonal $\mu_3$ in $H_K'$.  Because $H_K'$ is connected and in $E_K$, this $\mu_3$ is in $E_K$.  Hence $E_K$ is connected, has Lie algebra $\g(J_K)$ and contains $\mu_3$ in its center, which proves that $E_K \simeq E_6^{sc}$.\end{proof}

Denote by $\theta^{(4)}$ the automorphic function that is the pullback to $E_K$ via the embedding $E_K \rightarrow G_J$.  The following result proves Corollary \ref{cor1} (2).   
\begin{proposition}\label{prop:theta4} The automorphic function $\theta^{(4)}$ is a modular form on $E_6$ of weight $4$.  It has nonzero rank four Fourier coefficients.  However, $\theta^{(4)}$ is distinguished in the following sense: if $a(\omega)$ denotes the Fourier coefficient associated to $\omega \in W_{J_K}$ and $\omega$ is rank four, then $a(\omega) \neq 0$ implies that $q(\omega) = \kappa^2$ for some $\kappa \in K^\times$ with $\kappa^* = -\kappa$.\end{proposition}
Again, because the quaternionic adjoint group is connected \cite{thang}, there is literally no difference between the archimedean theory for the simply connected quaternionic $E_6$ and the adjoint form.  Thus, the definitions and results of \cite{pollackQDS}--which were proved in the adjoint case--apply immediately to the group $E_K$.
\begin{proof}[Proof of Proposition \ref{prop:theta4}] To see that $\theta^{(4)}$ is a modular form of weight $4$, again it suffices to check that $\mathcal{D}_4 \theta^{(4)} = 0$, which may be done by, e.g., applying the results of \cite[section 7]{pollackQDS}.  Alternatively, one can apply results of \cite{loke2}. The Fourier coefficients of $\theta^{(4)}$ are controlled by \cite[Theorem 7.3.1]{pollackLL}, which gives the result.

More precisely, suppose $\omega \in W_{J_K}$, and denote by $a(\omega)$ the Fourier coefficient of $\theta^{(4)}$ associated to $\omega$.  As explained in \cite[section 7]{pollackLL}, if $\omega \in W_{J_K}$ and $\eta \in B^2$, then $\omega + \eta$ can be regarded as an element of $W_J$, by applying the second Tits construction.  Then for $\omega \neq 0$,
\[a(\omega) = \sum_{\eta \in B^2}{a_{\theta_{Gan}}(\omega + \eta)}.\]
Again, the sum is only has finitely many nonzero terms.  It follows immediately from \cite[Theorem 7.3.1 part (1)]{pollackLL} that for $a(\omega)$ to be nonzero and $\omega$ rank four, we need $q(\omega) = \kappa^2$ for some $\kappa \in K^\times$ with $\kappa^* = -\kappa$.

To see that $\theta^{(4)}$ has nonzero rank four Fourier coefficients for our particular choice $\lambda = 1$, $S =1$, one may proceed as follows.  Suppose $K = \Q(\kappa)$ with $\kappa^*=-\kappa$.  Then set $\eta = \langle 1, -\frac{\kappa}{2}\rangle$ and $\omega = \left(0,1,0, \frac{\kappa^2}{4}\right)$.  Then
\[\omega + \eta = \left(0,(1, 1), (0, \frac{\kappa}{2}), \frac{\kappa^2}{4}\right)\]
is rank one in $W_J$. As $\omega$ is rank $4$, this completes the proof.  \end{proof}

\subsection{The integral modular form on $G_2$}\label{subsec:integralG2} Recall from above that $\epsilon$ denotes the automorphic function on $F_4^{an}$ that takes on just two values and is orthogonal to the constant function \cite{ganGross,ganGrossSavin}.  More precisely, as mentioned above and following \cite{elkiesGrossIMRN}, \cite{ganGross}, \cite{ganGrossSavin}, one has
\[\# F_4^{an}(\Q)\backslash F_4^{an}(\A)/F_4^{an}(\widehat{\Z})F_4^{an}(\R) = 2,\]
where $F_4^{an}$ is the stabilzer of the element $I=1_3 \in H_3(\Theta_0)$.  The two double cosets we denote by $U_I$ and $U_E$.  Here
\[E = \left(\begin{array}{ccc} 2 & \beta & \beta^* \\ \beta^* & 2 & \beta \\ \beta & \beta^* & 2 \end{array}\right)\]
with $\beta = \frac{1}{2}(-1 + e_1 + e_2 + e_3 + e_4 + e_5 + e_6 + e_7) \in \Theta_{0}$.  See e.g. \cite{elkiesGrossIMRN}.  The set $U_I$  has measure $\frac{91}{691}$ and the set $U_E$ has measure $\frac{600}{691}$ \cite{ganGross}. The function $\epsilon$ takes the value $\frac{691}{91}$ on $U_I$ and value $-\frac{691}{600}$ on $U_E$.

Denote by $\Delta$ Ramanujan's elliptic modular cusp form of weight $12$, and recall that we set 
\[F_{\Delta}(g) = \int_{F_4^{an}(\Q)\backslash F_4^{an}(\A)}{\theta_{Gan}((g,h))\epsilon(h)\,dh},\]
the $\theta$-lift of $\epsilon$ to $G_2$. The rank four Fourier coefficients of $F_{\Delta}$ are discussed in \cite{ganGrossSavin}, and it is explained there that $F_{\Delta}$ is a level one, weight $4$ modular form on $G_2$.  The constant term is essentially in \cite{ganGross, elkiesGrossIMRN}. Thus much of the following result is contained in \cite{ganGrossSavin} and \cite{ganGross}.

To state the Fourier expansion of $F_{\Delta}$, we make some notations.  Suppose $\omega_0 \in W_{F} = Sym^3(V_2)$.  As in \cite{ganSW}, define
\[\Omega_{I}(\omega_0) = \left\{(a,b,c,d) \in W_J^{rk=1} = (F \oplus J \oplus J^\vee \oplus F)^{rk=1}: \left(a,\frac{(b,I^\#)}{3},\frac{(c,I)}{3},d\right) = \omega_0\right\}\]
and similarly define
\[\Omega_{E}(\omega_0) = \left\{(a,b,c,d) \in W_J^{rk=1} = (F \oplus J \oplus J^\vee \oplus F)^{rk=1}: \left(a,\frac{(b,E^\#)}{3},\frac{(c,E)}{3},d\right) = \omega_0\right\}.\]
(In our normalization, $\omega_0$ is integral if it has coefficients in $\Z \oplus \frac{\Z}{3} \oplus \frac{\Z}{3} \oplus \Z$.)

The following is an immediate corollary of Theorem \ref{thm:thetaE8} and what has been said above.
\begin{corollary} The weight $4$, level one modular form $F_{\Delta}$ on $G_2$ is nonzero and has rational Fourier coefficients with bounded denominators.  If $\theta_{Gan}$ is normalized to have $a_{\theta}((1,0,0,0)) = 1$, then the constant term of $F_{\Delta}$ is
\[F_{\Delta,00}(g) = 24 |\nu(g)|^{5}\left(\Phi_\Delta(g) \frac{x^8}{8!} + \Phi_\Delta'(g) \frac{y^8}{8!}\right).\]
Here $\Phi_\Delta$ is the automorphic form on $\GL_2$ with\footnote{We are here using the definition of the automorphy factor $j(g,i)$ from \cite{pollackQDS}.  This automorphy factor is essentially the cube of the usual automorphy factor on $\GL_2$.} $j(g,i)^{4}\Phi_{\Delta}(g)$ equal to Ramanujan's $\Delta$ function.  Moreover, for $\omega_0$ as above,
\[a_{F_{\Delta}}(\omega_0) = \left(\sum_{\omega \in \Omega_{I}(\omega_0)}{\sigma_{4}(\Delta(\omega))}\right) - \left(\sum_{\omega \in \Omega_{E}(\omega_0)}{\sigma_4(\Delta(\omega))}\right).\]
\end{corollary}
\begin{proof} The key point is that the $\zeta(5)$ term drops out, because $\epsilon$ is orthogonal to the constant functions.  Everything else follows immediately from what has been said, and the fact \cite{elkiesGrossIMRN} that
\[\int_{F_4^{an}(\Q)\backslash F_4^{an}(\A)}{\Phi_{Kim}((g,h))\epsilon(h)\,dh} = 3\Phi_{\Delta}(g).\]
\end{proof}

\section{Fourier expansion of the Heisenberg Eisenstein series}\label{sec:FEeis} In this section, we prove results about the degenerate Heisenberg Eisenstein series on $G_J$.  In the first subsection, we give an abstract discussion of its Fourier expansion, not utilizing any of the archimedean results of \cite{pollackQDS}.  In the second subsection, we analyze the Fourier expansions of the special values of $E_J(g,s,\Phi;n)$ at $s=n+1$ when this point is in the range of absolute convergence for the Eisenstein series.  It is proved that the Fourier coefficients are Euler products.  Finally, in the third and fourth subsections, we analyze the constant term and rank one Fourier coefficients of these Eisenstein series directly. 

\subsection{Abstract Fourier expansion} In this subsection, we give the ``abstract'' Fourier expansion of a Heisenberg Eisenstein series $E(g,s)$.  Thus, we assume that $E(g,s) = \sum_{\gamma \in P(F)\backslash G(F)}{f(\gamma g,s)}$ for $f(g,s)$ a section in $Ind_P^G(|\nu|^{s})$, but we do not assume anything special about this section.

\begin{lemma} Suppose that $J \neq F$, so that $G = G_J$ is not $G_2$.  Then $P(F) \backslash G(F)/ P(F)$ has five elements, represented by $\{w_0=1,w_1, w_2, w_3, w_4\}$ with $w_1(E_{13}) \in e \otimes W_J$, $w_2(E_{13}) \in \h(J)^{0}$, $w_3(E_{13}) \in f \otimes W_J$ and $w_4(E_{13}) = E_{31}$.  The Lie algebra elements $w_1(E_{13})$ and $w_{3}(E_{13})$ are rank one when considered in $W_J$, and $w_2(E_{13})$ in the minimal orbit in $\h(J)^0$.  If $J = F$ so that $G_J = G_2$, then the double coset $P(F)\backslash G(F)\slash P(F)$ has four elements, and is represented by $\{w_0=1,w_1, w_3, w_4\}$, with these $w_i$ as above.\end{lemma}

\begin{proof} The orbit $G(F) E_{13}$ consists of the elements $X \in \g_J(F)$ spanning a minimal line; these are the set of nonzero $X \in \g_J(F)$ satisfying $[X,[X,y]] + 2B_{g}(X,y)X = 0$ for all $y \in \g(J)$.  The double coset $P(F) \backslash G(F) \slash P(F)$ is thus identified with the $P(F)$-orbits on the minimal lines in $\g_J$.

Denote by $h$ the element $\mm{1}{0}{0}{-1} \in \sl_2 \subseteq \g(J)$. By the Bruhat decomposition, the coset representatives $w_i$ for $P(F) \backslash G(F) \slash P(F)$ can be chosen to be normalizers of a maximal torus; thus we may assume that the vectors $w_i(E_{13})$ are eigenvectors for the one-parameter subgroup with Lie algebra spanned by $h$.  Consequently, we may assume that the elements $w_i(E_{13})$ only have one nonzero component of the $5$-grading on $\g(J)$ determined by the Levi subgroup $H_J$ of $P$; see \cite[section 4.3.1]{pollackQDS}.  In particular, representatives can taken to be in $F E_{13}$, $e \otimes W_J$, $F h \oplus \h(J)^0$, $f \otimes W_J$, and $F E_{31}$.  

It is easy to see that there exists $w_1, w_3,$ and $w_4$ as in the statement of the lemma, and that furthermore that there is one $M(F)$-orbit of such $w_i$'s.  Thus, there are at least four $P(F)$-orbits on the minimal lines in $\g_J(F)$ in all cases.  When $J=F$ so that $G_J = G_2$, these are the only orbits.  One can verify this by hand--it is because there are no rank two elements in $W_{J=F}$--or see, e.g., \cite[equation (9)]{jiangRallis}.

Now suppose $J = H_3(C)$, so that $G$ is not $G_2$.  Suppose $X \in Fh + \h(J)^0$ spans a minimal line.  We claim that $X \in \h(J)^0$.  Indeed, write $X = \mu h + \phi \in F h \oplus \h(J)^0$. By taking $y = E_{13}$ in $[X,[X,y]] + 2B_{g}(X,y)X = 0$, one sees that $\mu =0$.  Thus $X \in \h(J)^0$ as claimed.  The elements so obtained in $\h(J)^0$ are the minimal elements of this Lie algebra, which are the elements $H_J(F) n_{L}(e_{11})$.  In particular, the group $H_J(F)$ acts transitively on them.  This completes the proof of the lemma. \end{proof}

Via the lemma, we have $E(g,s) = \sum_{i=0}^{4}{E_i(g,s)}$, with $E_i(g,s) = \sum_{\gamma \in P(F)\backslash P(F)w_iP(F)}{f(\gamma g,s)}$.  Thus, $E_0(g,s) = f(g)$. For $G=G_2$, we understand $E_2(g,s) = 0$. We now write out more explicit expressions for the other $E_i$.

Recall that if $\ell \in W_J$ is a rank one line, there is associated to it a flag
\[W_J \supseteq (\ell)^{\perp} \supseteq W(\ell) \supseteq \ell \supseteq 0\]
with $W(\ell)$ a certain maximal isotropic subspace.  Precisely,
\[W(\ell) = \{ x \in W_J: \langle x,\ell \rangle = 0 \text{ and } \Phi_{x,\ell} = 0\}.\]

\begin{lemma}\label{lem:EisSum} Assume that $Re(s) >>0$ so that the sum defining $E(g,s)$ converges absolutely.  Then one has the following expressions for the $E_i(g,s)$.
\begin{enumerate} 
\item For each rank one line $\ell$ in $e \otimes W_J$, select $\gamma(\ell) \in G(F)$ with $\gamma(\ell) E_{13} \in \ell$.  Then 
\[E_1(g,s) = \sum_{\ell \subseteq W_J^{rk=1}}\sum_{\mu \in (\ell)^{\perp} N_0(F)\backslash N(F)}{f(\gamma(\ell)^{-1}\mu g,s)}.\]
\item For each minimal line $F \phi \subseteq \h(J)^0$, select $\gamma(\phi) \in G(F)$ with $\gamma(\phi) E_{13} \in F\phi$.  Then
\[E_2(g,s) = \sum_{F \phi \subseteq \h(J)^0 \text{ minimal}}\sum_{\mu \in (\mathrm{ker}(\phi) N_0(F))\backslash N(F)}{f(\gamma(\phi)^{-1}\mu g,s)}.\]
\item For each minimal line $F \ell \in f \otimes W_J$, select $\gamma(\ell) \in G(F)$ with $\gamma(\ell) E_{13} \in \ell$.  Then
\[E_3(g,s) = \sum_{\ell \subseteq W_J^{rk=1}}\sum_{\mu \in W(\ell)\backslash N(F)}{f(\gamma(\ell)^{-1} \mu g,s)}.\]
\item One has
\[E_4(g,s) = \sum_{\mu \in N(F)}{f(w_4^{-1} \mu g,s)}.\]
\end{enumerate} \end{lemma}
\begin{proof} The expression for $E_4(g,s)$ is clear.  The rest follows easily from what has already been said.  The only thing that must still be computed are the stabilizers in $N(F)$ of the minimal lines in $e \otimes W_J$, $\h(J)^0$, and $f \otimes W_J$.  And for this, it suffices to work on the level of Lie algebras.  We write $n = e \otimes x + c E_{13}$ for a typical element of the Lie algebra of $N(F)$.

We separate into cases.  For $E_1(g,s)$, write a typical rank one element of $e \otimes W_J$ as $e \otimes v$.  Then $[n, e \otimes v] = \langle x,v\rangle E_{13}$, thus verifying the expression for $E_1(g,s)$.  For $E_2(g,s)$, suppose that $\phi \in \h(J)^0$ spans a minimal line.  Then $[n, \phi] = [e \otimes x, \phi] = e \otimes \phi(x)$.  This gives the stated expression for $E_2(g,s)$.  Finally, suppose $f \otimes v \in f \otimes W_J$ is a rank one element.  Then $[n, f \otimes v] = c e \otimes v + \left(\langle x,v \rangle \frac{ef}{2} + \frac{1}{2} \Phi_{x,v}\right)$.  Thus $[n, f \otimes v] =0$ if and only if $c = 0$ and $x \in W(\ell)$, where $\ell = Fv$.  This completes the proof of the lemma. \end{proof}

We now consider the Fourier expansions of the $E_i(g,s)$ along $N(F)$; because the $E_i(g,s)$ are $P(F)$-invariant, this makes sense.  Fix an additive character $\psi: F\backslash \A \rightarrow \C^\times$.  For $v \in W_J$, define $\chi_v: N(F)\backslash N(\A) \rightarrow \C^\times$ as $\chi_v(n) = \psi(\langle v, \overline{n}\rangle)$, where $\overline{n}$ denotes the image of $n \in N/[N,N] \simeq W_J$.  We set
\[E_i^v(g,s) = \int_{N(F)\backslash N(\A)}{\chi_v^{-1}(n)E_i(ng,s)\,dn}.\]
Our measure is normalized so that if $U$ is a closed algebraic subgroup of $N$, then $[U]:=U(F)\backslash U(\A)$ has volume $1$.

Recall that elements of $W_J$ have a \emph{rank}, which is $0,1,2,3$ or $4$.
\begin{lemma}\label{lem:rkiv} If $\mathrm{rank}(v) > i$, then $E_i^v(g,s) = 0$.\end{lemma}
\begin{proof} Let us write
\begin{equation}\label{eqn:EiSum1} E_i(g,s) = \sum_{\ell}\sum_{\mu \in N_{\ell}(F)\backslash N(F)}{f(\gamma(\ell)^{-1} \mu g,s)}\end{equation}
for the expression given in Lemma \ref{lem:EisSum}.  Then
\begin{align*} E_i^v(g,s) &= \sum_{\ell} \int_{[N]}{\chi_v(n)^{-1} \left( \sum_{\mu \in N_{\ell}(F)\backslash N(F)}{f(\gamma(\ell)^{-1} \mu n g,s)}\right)\,dn} \\ &= \sum_{\ell} \int_{N(\ell)(F)\backslash N(\A)}{\chi_v^{-1}(n) f(\gamma(\ell)^{-1} n g,s)\,dn} \\ &= \sum_{\ell} \int_{N(\ell)(\A)\backslash N(\A)}{\left(\int_{[N(\ell)]}{\chi_v^{-1}(r)\,dr}\right) \chi_v^{-1}(n) f(\gamma(\ell)^{-1} n g,s)\,dn} \\ &= \sum_{\ell, \chi_v|_{N(\ell)} =1}\int_{N(\ell)(\A)\backslash N(\A)}{\chi_v^{-1}(n) f(\gamma(\ell)^{-1} n g,s)\,dn}.\end{align*}

It follows that $E_i^v(g,s)$ vanishes if $\chi_v$ restricted to $N(\ell)$ is nontrivial for all $\ell$ appearing in the sum (\ref{eqn:EiSum1}).  If $\mathrm{rank}(v) > i$, then it is not hard to see that this indeed happens.\end{proof}

\begin{lemma} If $i = \mathrm{rank}(v)$, then $E_i^v(g,s)$ is Eulerian.  More precisely,
\begin{enumerate}
\item Suppose $v$ is rank one.  Define $N_v^1 = (Fv)^\perp N_0 \subseteq N$.  Then
\[E_1^v(g,s) = \int_{N^1_v(\A)\backslash N(\A)}{\chi_v^{-1}(n)f(\gamma(\ell_v)^{-1} n g,s)\,dn}.\]

\item Suppose $v$ is rank two.  Define $N^2_v =\mathrm{ker}(\Phi_{v,v})N_0 \subseteq N$, and take $\gamma(\Phi_{v,v})$ with $\gamma(\Phi_{v,v})E_{13} = \Phi_{v,v} \in \h(J)^0$. Then
\[\int_{(N^2_v)(\A)\backslash N(\A)}{f(\gamma(\phi)^{-1}ng,s)\chi_v^{-1}(n)\,dn}.\]

\item Suppose $v$ is rank three.  Recall the element $v^\flat =t(v,v,v)$ (e.g., \cite[section 3.4.2]{pollackQDS}), which is rank one because $v$ is rank three. Define $N^3_v = W(F v^\flat) \subseteq N$, and take $\gamma(v^\flat)$ with $\gamma(v^\flat)E_{13} = f \otimes v^\flat$.  Then
\[E_3^v(g,s) = \int_{N^3_v(\A)\backslash N(\A)}{\chi_v^{-1}(n)f(\gamma(v^\flat)^{-1} ng,s)\,dn}.\]

\item Suppose $v$ is rank four.  Then 
\[E_4^v(g,s) = \int_{N(\A)}{\chi_v^{-1}(n)f(w_4^{-1} n g,s)\,dn}.\]
\end{enumerate}
\end{lemma}

\begin{proof} The case of $v$ rank one and $v$ rank four follow immediately from Lemma \ref{lem:rkiv}.

Consider first the case of $v$ rank two.  For $\phi \in \h(J)^0$ spanning a minimal line, define $I(\phi) = \ker{\phi}^\perp$.  Then $\chi_v$ is $1$ on $\ker{\phi}(\A)$ if and only if $v \in I(\phi)$.  Thus we have
\[E_2^v(g,s) = \sum_{F \phi \text{  with } v \in I(\phi)}\int_{(\mathrm{ker}(\phi)N_0)(\A)\backslash N(\A)}{f(\gamma(\phi)^{-1}ng,s)\chi_v^{-1}(n)\,dn}.\]
If $v$ is rank two, the only minimal line $F\phi$ with $v \in I(\phi)$ is $F \Phi_{v,v}$, so in this case $E_2^v(g,s)$ is Eulerian, as in the statement of the lemma.

Now consider the case of $v$ rank three. One has that $\chi_v$ is $1$ on $W(\ell)(\A) \subseteq N(\A)$ if and only if $v \in W(\ell)^\perp = W(\ell)$.  Thus
\[E_3^v(g,s) = \sum_{\ell \text{ rank one with } v \in W(\ell)}\int_{W(\ell)(\A)\backslash N(\A)}{\chi_v^{-1}(n)f(\gamma(\ell)^{-1} n g,s)\,dn}.\]
Thus, if $v$ is rank three, then $F\ell = F v^\flat$, so the line $\ell$ is determined by $v$ and $E_3^v(g,s)$ is Eulerian as specified above.  This completes the proof of the lemma. \end{proof}

For the constant term $E_1^0(g,s)$, we record now that
\[E_1^0(g,s) = \sum_{\ell \subseteq W_J^{rk =1}}{\int_{N^1_{\ell}(\A)\backslash N(\A)}{f(\gamma(\ell)^{-1} ng,s)\,dn}}.\]
This function is as an Eisenstein series for the Levi subgroup $H_J$ attached to the parabolic that stabilizes a rank one line in $W_J$.

\subsection{Euler product} The purpose of this subsection is to prove that the special (modular form) values of the Heisenberg Eisenstein series have non-constant Fourier coefficients that are Euler products.  More precisely, for $v \in W_J$ set
\[E^v(g,s;n) = \int_{N(F)\backslash N(\A)}{\chi_v^{-1}(u)E(ug,s;n)\,du}.\]
We know from the previous section that if $v \in W_J$ is nonzero, then $E^v(g,s;n) = \sum_{i \geq rk(v)}{E_i^v(g,s;n)}$, and that the term $E_{rk(v)}^v(g,s;n)$ is an Euler product.  In this subsection we prove that if the special value $s=n+1$ is in the range of absolute convergence for the Eisenstein series, then the terms $E_i^v(g,s,n)$ vanish at this point when $i > \mathrm{rank}(v)$.  

\begin{proposition}\label{prop:Eisivan} Suppose $n$ is even, and $n > \dim(J)+1$ so that the sum defining $E(g,s;n)$ converges absolutely at $s=n+1$ and defines a modular form of weight $n$ at this point.  Then if $v \neq 0$ and $i > \mathrm{rank}(v)$, $E_{i}^{v}(g,s;n)$ vanishes at $s=n+1$.\end{proposition}

We will prove Proposition \ref{prop:Eisivan} by applying the following corollary of the main result of \cite{pollackQDS}.

\begin{corollary}\label{cor:FwEquiv} Suppose $\omega \in W_J$ is nonzero, $m \in H_J(\R)$, and $F$ is a modular form on $G_J$ of weight $n$.  Denote by $F_{\omega}$ the $\omega$ Fourier coefficient of $F$, and set $\tilde{m} = \nu(m) m^{-1}$.  Then if $\tilde{m} \omega = \omega$, $F_{\omega}(mg) = \nu(m)^{n} |\nu(m)| F_{\omega}(g)$.  For the constant term $F_{00}$, one has the following: $F_{00}$ is the sum of two terms $f_1$ and $f_2$, which are distinguished by the following properties.  For $z \in Z_{H}\simeq \GL_1(\R)$, one has $f_1(zg) = z^{2n+2} f_1(g)$ and $f_2(zg) = z^{n+2}f_2(g)$. \end{corollary}

\begin{remark} We have derived this corollary as a consequence of the complete formula given in \cite{pollackQDS}.  However, the corollary is useful by itself.  Can the corollary be proved directly, by a ``softer'' method?\end{remark}

The idea of the proof of Proposition \ref{prop:Eisivan} is to show that $E_i^v(g,s)$ is an absolutely convergent sum of terms, each of which vanishes at $s=n+1$ because it does not satisfy the required equivariance property enforced by Corollary \ref{cor:FwEquiv}.  We now make some preparatory remarks that we will use to carry out this strategy.

For $\omega \in W_J$, recall that we denote $\chi_{\omega}(n) = \psi(\langle \omega, n\rangle)$.  Suppose we have an absolutely convergent integral
\[f^{\omega}(g,s) = \int_{Y(\R)}{f(w^{-1} y g,s)\chi_\omega(y)^{-1}\,dy},\]
and $m \in H_J$ with $\widetilde{m} \omega = \omega$, and $m$ preserves $Y$.  Here $Y \subseteq N$.  Furthermore, assume $f(p g,s) = |\nu(p)|^{s}f(g,s)$ for $p \in P$ the Heisenberg parabolic, and that $w^{-1}mw \in P$.  Then we have

\[f^{\omega}(mg,s) = J(m,Y) |\nu(w^{-1} m w)|^{s} f^{\omega}(g,s)\]
where
\[J(m,Y) = \left|\frac{d(m y m^{-1})}{dy}\right|_{Y}\]
is the measure term that comes from a change of variable in the integral.  

For a rank one line $\ell \subseteq \g(J)$, and $g \in G_J$ stabilizing $\ell$, define $\mu_\ell(g) \in \GL_1$ to be the element with $g \ell = \mu_{\ell}(g) \ell$.  Then $\nu(w^{-1} m w) = \mu_{w E_{13}}(m)$.  Thus, with assumptions and notations as above, we have
\[f^{\omega}(m g,s) = J(m,Y) |\mu_{w E_{13}}(m)|^{s} f^{\omega}(g,s).\]

We will now analyze term $C(m,Y,s) :=J(m,Y)|\mu_{w E_{13}}(m)|^{s}$ for the integrals that appear in the previous subsection.  We first consider the constant term, i.e., $\omega = 0$.
\begin{enumerate}
\item The case $i=0$.  We have $f_0(zg,s) = |\nu(z)|^{s}f_0(g,s) = |z|^{2s} f_0(g,s)$.  At $s=n+1$, this character is $z^{2n+2}$.
\item The case $i=1$.  Fix $\ell \in e \otimes W_J$ rank one, and set $Y_1 = N_{\ell}^1\backslash N$. We have $J(z,Y_1) = |z|$.  Because $\gamma(\ell) E_{13} \in e \otimes W_J$, $\mu_{\gamma(\ell) E_{13}}(z) = z$.  Thus, $C(z,Y_1,s) = |z|^{s+1}$, which gives $z^{n+2}$ at $s=n+1$.
\item The case $i=2$.  Fix $\phi \in h(J)^0$ spanning a minimal line, and set $Y_2 = \ker(\phi)N_0\backslash N$.  Then $J(z,Y_2) = |z|^{4+dim(C)}$, where recall $J=H_3(C)$.  In this case, $\gamma(\phi) E_{13} \in \h(J)^0$, so $\mu_{w}(z) = 1$.  Thus $C(z,Y_2^0,s) = |z|^{4+\dim(C)}$, independent of $s$.
\item The case $i=3$.  Fix $\ell \subseteq f \otimes W_J$ a rank one line, and set $Y_3 = W(\ell)\backslash N$.  In this case, $J(z,Y_3) = |z|^{\dim(J)+3}$. (Note that $z$ scales the measure on $N_0$ by $|z|^2$.) Furthermore, $\gamma(\ell) E_{13} \in f \otimes W_J,$ and thus $\mu_{w E{13}}(z) = z^{-1}$.  Hence $C(z,Y_3,s) = |z|^{\dim(J)+3-s}$, which specializes to $|z|^{\dim(J)+2-n}$ at $s=n+1$.
\item The case $i=4$.  Set $Y_4 = N$.  Then $J(z,Y_4) = |z|^{\dim W_J + 2} = |z|^{2\dim(J)+4}$. We have $\mu_{w_4 E_{13}}(z) = z^{-2}$, thus $C(z,Y_4,s) = |z|^{2\dim(J)+4-2s}$, which at $s=n+1$ is $|z|^{2\dim(J)+2-2n}$.  
\end{enumerate}

We obtain the following lemma.  Recall our running assumption that $n > 0$ is even.
\begin{lemma}\label{lem:Econstz} Suppose that the meromorphic continuation of $E(g,s,n)$ to $s=n+1$ is regular and defines a modular form of weight $n$ at this point.  Suppose moreover that $n$ and $J$ are such that the numbers $2n+2,n+2,4+\dim(C), \dim(J) + 2 -n$, and $2\dim(J)+2-2n$ are pairwise distinct. If $G=G_2$ then exclude the term $4 +\dim(C)$. Then only $E_0(g,s)$ and $E_1(g,s)$ contribute to the constant term of $E(g,s,n)$ at $s=n+1$.\end{lemma}

Note that in Lemma \ref{lem:Econstz}, we do not need to assume that the Eisenstein series $E(g,s,n)$ converges absolutely at $s=n+1$.

\begin{proof}[Proof of Lemma \ref{lem:Econstz}] If $E(g,s,n)$ is regular at $s=n+1$, then so is its constant term.  Each $E_i(g,s)$ contributes to the constant term.  Now, under the assumptions of the lemma, the constant terms of the $E_i(g,s)$ transform by different characters under left-translation $g \mapsto zg$ for $z \in Z_H(\R)$.  It follows the constant terms of the $E_i(g,s)$ are each individually regular at $s=n+1$.  Because they are regular at $s=n+1$, the values of these constant terms at $s=n+1$ must transform by the characters $\{2n+2,n+2,4+\dim(C),\dim(J)+2-n,2\dim(J)+2-2n\}$.  However, the constant term of a modular form can only transform by $z^{2n+2}$ and $z^{n+2}$.  Again because these numbers are assumed all distinct, only $E_0^0(g,s)$ and $E_1^0(g,s)$ contribute; the other $E_i^0(g,s)$ for $i=2,3,4$ must vanish at $s=n+1$.\end{proof}

We consider next the nonconstant Fourier coefficients.  For the nonconstant terms, we will need to assume that the sum defining the Eisenstein series converges absolutely.  Let the $Y_i$ be as above.  

We will use the action of a certain four-dimensional torus $T \subseteq H_J$ when $J=H_3(C)$.  We parametrize the elements of $T$ as ordered four-tuples $t=(\lambda,t_1,t_2,t_3)$.  In these coordinates, set $\delta(t) :=t_1t_2t_3$.  Then $T$ acts on $W_J$ as
\[t (a,b,c,d) = (\lambda^{-1} \delta^{-1}(t) a, \delta^{-1}(t) (t \cdot b), \lambda \delta(t) (t^{-1} \cdot c), \lambda^2 \delta(t) d)\]
where $t \cdot X = \diag(t) X \diag(t)$ for $X \in H_3(C)$.  Moreover, one has $\nu(t) = \lambda$ and $Ad(t) n_L(X) = \lambda n_L( t \cdot X)$.  

We set $T^1 \subseteq T$ the subtorus consisting of elements with $\lambda =1$.  Additionally, define $T' \subseteq T$ the subtorus consisting of those elements where $t_1=t_2=t_3$.  Then when $G=G_2$, $T' \subseteq P \subseteq G_2$ and acts on $W_J$ by the same formula.  We are now ready to prove Proposition \ref{prop:Eisivan}.

\begin{proof}[Proof of Proposition \ref{prop:Eisivan}] Because of the assumption that the Eisenstein series converges absolutely, we may analyze the terms $E_i^{\omega}(g,s)$ separately for each $i$.  Assume $\omega \neq 0$.

{\bf The case $i=4$} Consider first the case $i=4$.  Suppose $m \in H_J$. Then $J(m,Y_4) = |\nu(m)|^{\dim J + 2}$, and $\mu_{w_4 E_{13}}(m) = \nu(m)^{-1}$.  Thus $C(m,Y,s) = |\nu(m)|^{\dim J + 2 - s}$, which specializes to $\dim(J) +1 -n$ at $s=n+1$.  Note that, if $\omega$ rank four and $\widetilde{m} \omega = \omega$, then necessarily $|\nu(m)| =1$.  By contrast, if $\omega$ is rank $1,2$ or $3$, there there exists $m \in M$ with $\widetilde{m} \omega = \omega$ but $|\nu(m)| \neq 1$.  Indeed, if $\omega$ is rank two or three, then $\omega$ has an $H_J(F)$-translate $\omega'$ of the form $(0,*,0,0)$, and if $\omega$ is rank one, then $\omega$ has a translate $\omega'$ of the form $(1,0,0,0)$.  In the first case, $\omega'$ is fixed by $t = (\lambda, (1,1,1))$, while in the second, $\omega'$ is fixed by $(\lambda^3,(\lambda^{-1},\lambda^{-1},\lambda^{-1}))$.  Thus, so long as $\dim J + 1 -n \neq n+1$, the term $E_4^\omega(g,s)$ will vanish for $\omega$ not rank $4$.  However, the condition $\dim(J) = 2n$ is never satisfied for an even $n$, as the allowable $\dim(J)$'s are $\{1,6,9,15,27\}$.

{\bf The case $i=3$} Now consider the case $i=3$.  Recall that we are analyzing
\[E_3^\omega(g,s) = \sum_{\ell: \omega \in W(\ell)}\int_{W(\ell)(\A)\backslash N(\A)}{\chi_\omega^{-1}(n)f(\gamma(\ell)^{-1}ng,s)\,dn}\]
where $\gamma(\ell) E_{13}$ spans $\ell \subseteq f \otimes W_J$.  We are interested in seeing if this vanishes at $s=n+1$ for $\omega$ of rank one or two.

Because of the absolute convergence, we may consider each term separately, and then without loss of generality we may assume that $\ell$ is spanned by $f \otimes (0,0,0,1) = E_{21}$, so that $W(\ell) = (0,0,*,*)$.  Recall that the condition on $\omega$ is $\omega \in W(\ell)$. If $G = G_2$, then necessarily $\omega$ is rank one and in $(0,0,0,*)$.  If $G$ is not $G_2$, then without loss of generality we may assume that $\omega = (0,0,\omega_2,0)$, with $\omega_2 \in e_{11} \times J \subseteq J^\vee$, i.e., $\omega_2$ has $(i,j)$ coefficient equal to $0$ if either $i$ or $j$ equals $1$.  We set $Y_3 = W(\ell)\backslash N$ as above.

Suppose first that $G=G_2$, and $t' = (\lambda,(t,t,t)) \in T'$.  Then with notation as above, $J(t',Y_2) = |t|^{-4}$.  Furthermore, one has $\mu_{E_{21}}(t') = \lambda t^3$.  Thus, $C(t',Y_3,s) = |\lambda t^3|^{s} |t|^{-4} = |\lambda t^3|^{n+1}|t|^{-4}$ at $s=n+1$.  Now, $\widetilde{t'} (0,0,0,1) = \lambda^{-1}t^{-3}(0,0,0,1)$, and thus the condition $\widetilde{t'} \omega = \omega$ is $\lambda t^3 = 1$.  Thus for such $t'$, $C(t',Y_3,s=n+1)=|\nu(t')|^{n+1}$ if and only if $|t|^{-4} = |\lambda|^{n+1} = |t|^{-3n-3}$.  But $4 \neq 3n+3$, and we can find $t'$ with $\widetilde{t'}\omega = \omega$ and $|t| \neq 1$.  Thus, in case $G=G_2$, this term vanishes.

Now suppose that $G \neq G_2$.  Consider the action of the subgroup of $T^1$ with $t_1 = 1$ and $t_2 = t_3$.  Then this group fixes $\omega$.  One computes that $J(t,Y_3) = |\delta|^{-1-\dim J/3}$.  The term $\mu_{\ell}(t) = \delta$, thus $C(t,Y_3,s) = |\delta|^{s-1-\dim J/3}$.  Because $\nu(t) = 1$ but the elements $(1,t,t) \in T$ have $\delta \neq 1$, this only transforms the right way at $s=n+1$ if $3n=\dim(J)$.  But $n > 2$ is even, so this cannot occur for $\dim(J) \in \{6,9,15,27\}$.  This completes the proof of the vanishing for $i=3$.

{\bf The case $i=2$}  We are now left with the case of  $i=2$.  This case cannot occur when $G=G_2$, so we assume $J = H_3(C)$ so that $G\neq G_2$.  Recall that we are interested in
\[E_2^{\omega}(g,s) = \sum_{\phi: \omega \in I(\phi)}\int_{(\ker(\phi)N_0)(\A)\backslash N(\A)}{\chi_\omega^{-1}(n)f(\gamma(\phi)^{-1}ng,s)\,dn}.\]
We'd like to see that this vanishes at $s=n+1$ if $\omega$ is rank one.  If $\phi$ is fixed, set $Y_2 = \ker(\phi)N_0\backslash N$.

By absolute convergence, we may consider the individual terms separately.  Thus we may assume that $\phi = n_{L}(e_{11})$.  The condition on $\omega$ is that $\omega \in I(\phi) = (0,Fe_{11}, e_{11} \times J, F)$.  Without loss of generality, we may furthermore assume that $\omega = (0,0,0,1)$.  Note that the subgroup of $T^1$ with $\delta = t_1 t_2 t_3 = 1$ fixes $\omega$.

Now, for $t \in T^1$, we have $\mu_{F \phi}(t) = t_1^2$.  Moreover, one computes $J(t,Y_2) = |t_1|^{-(\dim(C) + 4)}$.  Thus, $C(t,Y_2,s) = |t_1|^{2s-\dim(C)-4}$, which is $|t_1|^{2n-\dim(C)-2}$ at $s=n+1$.  However, $\nu(t) = 1$, so if $|t_1| \neq 1$, this transforms the correct way only if $2n=\dim(C)+2$.  But then $n < \dim(J)$, so this cannot occur.  This completes the argument in the case $i=2$.
\end{proof}

Summarizing, we have proved the following result.
\begin{theorem}\label{thm:FEconvEis} Suppose $n > \dim(J)+1$ is even, so that $E(g,s;n)$ converges absolutely, and defines a modular form at $s=n+1$.  Then, the Fourier expansion of $E(g,s;n)$ at $s=n+1$ is given as follows:
\begin{enumerate}
\item If $\omega$ is rank one, then 
\[E^{\omega}(g,s=n+1;n) = \int_{((F\omega)^\perp N_0)(\A)\backslash N(\A)}{\chi_{\omega}^{-1}(u)f(\gamma(\omega)^{-1}ug,s=n+1)\,du}\]
where $\gamma(\omega)E_{13} = e\otimes \omega$.
\item If $\omega$ is rank two (which cannot occur in case $G=G_2$), then
\[E^{\omega}(g,s=n+1;n) = \int_{(\ker(\Phi_{\omega,\omega}) N_0)(\A)\backslash N(\A)}{\chi_{\omega}^{-1}(u)f(\gamma(\Phi_{\omega,\omega})^{-1}ug,s=n+1)\,du}\]
where $\gamma(\Phi_{\omega,\omega})E_{13} = \Phi_{\omega,\omega}.$
\item If $\omega$ is rank three, then
\[E^{\omega}(g,s=n+1;n) = \int_{(W(\omega^\flat)(\A)\backslash N(\A)}{\chi_{\omega}^{-1}(u)f(\gamma(\omega^\flat)^{-1}ug,s=n+1)\,du}\]
where $\gamma(\omega^\flat)E_{13} = f \otimes \omega^\flat.$
\item If $\omega$ is rank four, then
\[E^{\omega}(g,s=n+1;n) = \int_{N(\A)}{\chi_\omega^{-1}(u)f(w_4^{-1}ug,s=n+1)\,d}u.\]
\end{enumerate}
Finally, the constant term of $E(g,s=n+1;n)$ is $f_0(g,s=n+1) + E_1^0(g,s=n+1;n)$, where
\[E_1^0(g,s=n+1;n) = \sum_{\ell \subseteq W_J^{rk 1}}{\int_{(\ell)^\perp N_0(\A)\backslash N(\A)}{f(\gamma(\ell)^{-1}ug,s=n+1)\,du}}\]
and $\gamma(\ell)E_{13}$ spans $e \otimes \ell$.\end{theorem}

\subsection{Computation of constant term} We now compute the constant terms more explicitly. 

\subsubsection{The $i=0$-term} The first thing that we do is compute the simplest term, $f_0(g,s)$.
\begin{lemma} Suppose $g \in P_{Heis}(\A)$.  Then
\[f_0(g,s) = |\nu(g)|^{s}\zeta(s) \frac{(-1)^{n/2}}{2^n} \Gamma_\R(s+n) x^n y^n,\]
which at $s = n+1$  becomes
\[f_0(g,s=n+1) = |\nu(g)|^{n+1} \zeta(n+1) \frac{(-1)^{n/2}}{2^n} \pi^{-n} (1/2)_n x^n y^n.\]
\end{lemma}
\begin{proof}  Assume $g \in P \subseteq G_J$ the Heisenberg parabolic.  Then we have
\begin{align*} f_0(g,s) &= \int_{\GL_1(\A)}{|t|^{s} \Phi(t g^{-1} E_{13})\,dt} \\ &= \int_{\GL_1(\A)}{|t|^{s} \Phi(t \nu(g)^{-1} E_{13})\,dt} \\ &= |\nu(g)|^{s} \int_{\GL_1(\A)}{|t|^{s}\Phi(t E_{13})\,dt}.\end{align*}
Thus, the contribution to $f_0(g,s)$ from the finite places is $|\nu(g)|^{s} \zeta(s)$.  

Let us now analyze the archimedean section $f(g,\Phi_n,s)$.  Observe that for $v \in \g(J)$,
\[pr(v) = B(v,f_{\ell}) e_{\ell} + \frac{1}{2} B(v,h_{\ell}) h_{\ell} + B(v,f_{\ell}) e_{\ell}.\]
Also, recall \cite[section 5.1]{pollackQDS}
\begin{itemize}
\item $e_{\ell} = \frac{1}{4}(ie+f) \otimes r_0(i)$,
\item $f_{\ell} = \frac{1}{4}(ie-f) \otimes r_0(-i)$, and 
\item $h_{\ell} = \frac{i}{2}\left(\mm{}{1}{-1}{} + n_{L}(-1) + n_{L}^\vee(1)\right)$.
\end{itemize}
and under the map $pr: \k \rightarrow Sym^2(V_2)$, $e_{\ell} \mapsto x^2$, $f_{\ell} \mapsto -y^2$, and $h_{\ell} \mapsto -2xy$ \cite[section 9]{pollackQDS}.  Consequently,
\[ pr(E_{13}) = -\frac{i}{4} h_{\ell} \mapsto \frac{i}{2} xy.\]
Thus if $n$ is even, 
\begin{equation}\label{eqn:fn1}f_{\infty}(1,\Phi_{\infty,n},s) = pr(E_{13})^{n} \int_{\GL_1(\R)}{|t|^{s+n} e^{-\pi t^2}\,dt} = \frac{(-1)^{n/2}}{2^n} \Gamma_{\R}(s+n) x^n y^n.\end{equation}

Combining the finite places and the archimedean place, we obtain the lemma.\end{proof}

\subsubsection{The $i=1$-term} Denote by $\ell_0$ the line spanned by $E_{23} = e \otimes (0,0,0,1)$, and suppose $\gamma_0 E_{13} = E_{23}.$  Define
\[f_1^0(g,s) = \int_{((\ell_0)^\perp N_0)(\A)\backslash N(\A)}{f(\gamma_0^{-1} n g,s)\,dn}.\]
Then 
\[E_1^0(g,s) = \sum_{\gamma \in P_{Sieg}(\Q)\backslash H_J(\Q)}{f_1^0(\gamma g,s)}\]
where $P_{Sieg}$ denotes the parabolic subgroup of $M$ that stabilizes the line $\Q(0,0,0,1)$. 
\begin{proposition} Suppose that $\Phi_v$ restricted to $F E_{13}\oplus FE_{23}$ is the characteristic function of $\Z_v E_{13} \oplus \Z_v E_{23}$ for every $v < \infty$.  Then for $p \in P_{Sieg}(\A)$,
\[ f_1^0(p,s=n+1) = \frac{\zeta(n)\Gamma(n)}{(4\pi)^n}|\nu(p)||\lambda(p)|^{n}(x^{2n}+y^{2n})\]
with $\lambda(p)$ defined by $p (0,0,0,1) = \lambda(p)(0,0,0,1)$.\end{proposition}
See also \cite[section 13]{ganSW} and especially \cite[Lemma 13.14]{ganSW} where the case of $G_2$ is discussed.
\begin{proof} We have
\begin{align*} f_1^0(g,s) &= \int_{((\ell_0)^\perp N_0)(\A)\backslash N(\A)}{f(\gamma_0^{-1} n g,s)\,dn}\\ &= \int_{\GL_1}\int_{F}{|t|^{s}\Phi(t g^{-1} n^{-1}E_{23})\,dt\,dn} \\ &= \int_{\GL_1}\int_{F}{|t|^{s}\Phi(t g^{-1} (E_{23}-x E_{13}))\,dt\,dx}.\end{align*}
Here we have used that if $n=e \otimes (x,0,0,0)$, then $n^{-1} E_{23} = E_{23} - x E_{13}$.  

Now suppose $g = p \in P_{Sieg} \subseteq H_J$, the Siegel parabolic subgroup of the Levi $H_J$.  Denote by $\lambda$ the character of $P_{Sieg}$ that defines its Siegel Eisenstein series, i.e. $p E_{23} = \lambda(p) E_{23}$.  Then $p^{-1}(E_{23}-x E_{13}) = \lambda(p)^{-1} E_{23} - x \nu(p)^{-1} E_{13}$.  So, we must evaluate
\begin{align*} f_1^0(g,s) &= \int_{\GL_1}\int_{F}{|t|^{s}\Phi(t \lambda(p)^{-1} E_{23} - x t\nu(p)^{-1} E_{13})\,dx\,dt} \\ &= |\nu(p)||\lambda(p)|^{s-1}\int_{\GL_1}\int_{F}{|t|^{s}\Phi(t E_{23} - xt E_{13})\,dx\,dt}\end{align*}
where we have changed variables $x \mapsto \nu(p)\lambda(p)^{-1} x$ and $t \mapsto \lambda(p) t$.

Now, assuming that $\Phi_v$ is as in the statement of the proposition, the integral over $t$ and $x$ gives $\zeta_v(s-1)$.  Thus, at the finite place $v$ with $\Phi_v$ as above and $p \in P_{Sieg} \subseteq H_J$, we get 
\[f_1^0(p,s) = \zeta_v(s-1) |\nu(p)| |\lambda(p)|^{s-1}.\]

We now must calculate what happens at the archimedean place for $f_1^0(g,s)$.  We require an explicit expression for $pr(u)$ for $u = e \otimes v + \mu E_{13}$.  We have $B(u,e_{\ell}) = \frac{1}{4} \langle v, r_0(i)\rangle$, $B(u,f_{\ell}) = -\frac{1}{4} \langle v, r_0(-i)\rangle$, and $B(u,h_{\ell}) = -\frac{i}{2}\mu$.  Thus
\begin{align*} -4 pr(e \otimes v + \mu E_{13}) &= \langle v, r_0(-i)\rangle e_{\ell} + i \mu h_{\ell} - \langle v, r_0(i) \rangle f_{\ell} \\ &\mapsto \langle v, r_0(-i) \rangle x^2 -2i \mu x y + \langle v, r_0(i)\rangle y^2.\end{align*}
Therefore
\[pr( E_{23} - \beta E_{13}) = -\frac{1}{4}(-x^2+2 i \beta xy - y^2) = \frac{1}{4}(x^2-2i\beta xy + y^2).\]

Thus
\begin{align*} f_1^0(g,s) &= \int_{\GL_1}\int_{F}{|t|^{s} \Phi_{\infty,n}(t(E_{23}-\beta E_{13}))\,dx\,dt} = \int_{\GL_1}\int_{F}{|t|^{s+n} pr( E_{23}-\beta E_{13})^n e^{-\pi t^2(1 + \beta^2)}\,d\beta\,dt} \\ &=\frac{1}{4^n} \int_{\GL_1}\int_{F}{|t|^{s+n} (x^2 - 2i\beta xy + y^2)^n e^{-\pi t^2(1 + \beta^2)}\,d\beta\,dt}.\end{align*}
The final integral is
\[\frac{1}{4^n} \sum_{0 \leq k \leq n, \text{ even}} \binom{n}{k} 2^k (-1)^{k/2}(x^2+y^2)^{n-k}(xy)^k \int_{\GL_1}{|t|^{s+n} e^{-\pi t^2} \left(\int_{\R}{\beta^k e^{-\pi t^2 \beta^2}\,d\beta}\right)\,dt}.\]
The inner $\beta$-integral above gives
\begin{align*} \int_{\R}{\beta^k e^{-\pi t^2 \beta^2}\,d\beta} &= |t|^{-k-1} \int_{\R}{\beta^k e^{-\pi \beta^2}\,d\beta} \\ &= \pi^{-k/2} |t|^{-k-1} \left(\frac{1}{2}\right)_{k/2}\end{align*}

Therefore, we obtain
\begin{align*} |\nu(p)|^{-1}|\lambda(p)|^{1-s} f_1^0(p,s) &= \frac{1}{4^n} \sum_{0 \leq k \leq n, \text{ even}} \binom{n}{k} 2^k (-1)^{k/2}(x^2+y^2)^{n-k}(xy)^k \pi^{-k/2} \left(\frac{1}{2}\right)_{k/2} \\ &\;\;\; \times \Gamma_{\R}(s+n-k-1) \\ &\stackrel{s=n+1}{=} \frac{\pi^{-n}}{4^n} \sum_{0 \leq k \leq n, \text{ even}} \binom{n}{k} 2^k (-1)^{k/2}(x^2+y^2)^{n-k}(xy)^k \left(\frac{1}{2}\right)_{k/2} \Gamma(n-k/2)\end{align*}

\begin{lemma} One has
\[\sum_{0 \leq k \leq n, \text{ even}} \binom{n}{k} 2^k (-1)^{k/2}(x^2+y^2)^{n-k}(xy)^k \left(\frac{1}{2}\right)_{k/2} \Gamma(n-k/2) = \Gamma(n)(x^{2n}+y^{2n}).\]
\end{lemma}
\begin{proof} The proof is by generating series.  First, write $u = x^2$, $v = y^2$, $n = 2m$ and $k = 2j$.  Then we must evaluate
\[\sum_{0 \leq j \leq m}{\binom{2m}{2j}2^{2j}(-1)^j (u+v)^{2m-2j} (uv)^j \left(\frac{1}{2}\right)_j \Gamma(2m-j)}.\]
Because $\left(\frac{1}{2}\right)_j = \frac{(2j)!}{2^{2j} (j!)}$, this is
\[n! \sum_{0 \leq j \leq m}{(-1)^j \binom{2m-j}{j} \frac{1}{2m-j} (u+v)^{2m-2j}(uv)^{j}}.\]
To evaluate this, ignore the $n!$ and sum over all $n$ (not just even $n$) to obtain
\begin{align*} \sum_{n \geq 1}\sum_{0 \leq j \leq n-j}{ (-1)^j \binom{n-j}{j} \frac{1}{n-j} (u+v)^{(n-j)-j}(uv)^{j}} &= \sum_{p \geq 1}\sum_{0 \leq j \leq p}{(-1)^j \binom{p}{j} \frac{1}{p}(u+v)^{p-j}(uv)^j} \\ &= \sum_{p \geq 1}{\frac{((u+v)-uv)^{p}}{p}}.\end{align*}
But this last sum is
\[-\log(1-u-v+uv) =-\log(1-u) -\log(1-v) = \sum_{n \geq 1}{\frac{u^n + v^n}{n}}.\]
This proves the lemma. \end{proof}

We have thus proved at $s=n+1$ the archimedean section gives
\[f_1^0(p,s=n+1) = |\nu(p)||\lambda(p)|^{n} (4\pi)^{-n} \Gamma(n)(x^{2n} + y^{2n}).\]
This completes the proof of the proposition.\end{proof}

\subsection{The rank one Fourier coefficients} While the rank two, three and four coefficients of $E(g,\Phi,s=n+1)$ appear to be somewhat difficult to evaluate, the Fourier coefficients $E_1^\omega(g,s=n+1)$ for $\omega$ rank one can be computed directly.  In this subsection, we make this computation.

For $\omega \in W_J$ rank one, define
\[f_1^\omega(g,s) = \int_{(F\omega)^\perp N_0\backslash N}{\chi_\omega^{-1}(n)f(\gamma(\omega)^{-1}ng,s)\,dn}\]
where $\gamma(\omega) E_{13} = e \otimes \omega$.  We will assume that $\omega = a e \otimes (0,0,0,1) = a e\otimes \omega_0$ with $a \in F$.  So, $\omega_0 = (0,0,0,1)$ in this subsection.

\subsubsection{Evaluation at spherical finite places} First we evaluate $f_1^\omega(g,s)$ at spherical finite places. The Fourier coefficient will vanish unless $a \in \mathcal{O}$, so we can assume $a \in \mathcal{O}  = \Z_p$.  We get, as above,
\begin{align*} f^\omega_1(g,s) &= \int_{(\ell)^\perp N_0\backslash N}{\chi_{\omega}^{-1}(n) f(\gamma_0^{-1} ng,s)\,dn} \\ &= \int_{\GL_1}\int_{F}{\psi^{-1}(ax) |t|^{s}\Phi(t g^{-1}(E_{23} + xE_{13}))\,dx\,dt}.\end{align*}
(We have changed variables $x \mapsto -x$; note that $\chi_\omega(n) = \psi( \langle \omega, n\rangle) = \psi(-ax)$.)  Now, because we are interested in the Fourier expansion of the spherical vector, we will take $g = 1$ at the finite places. Thus, at the finite places, we obtain 
\begin{align*} f_1^\omega(1,s) &= \int_{\GL_1}\int_{F}{|t|^{s} \psi^{-1}(ax)\Phi(t E_{23} + t x E_{13})\,dx\,dt} \\ &= \int_{\GL_1}\int_{F}{|t|^{s-1} \psi^{-1}(ax/t)\Phi(t E_{23} +  x E_{13})\,dx\,dt} \\ &= \sum_{t|a}{|t|^{s-1}}.\end{align*}

\subsubsection{Evaluation at the archimedean place} We now evaluate $f_1^\omega(g,s)$ at the archimedean place.  Suppose $\omega_0 \in W_J$ is rank one, $z \in W_J$, $n = n(z) = \exp( e\otimes z)$, and $g \in H_J$.  Then
\begin{align*} g^{-1}n^{-1} e\otimes \omega_0 &= g^{-1}(e\otimes \omega_0 + \langle \omega_0, z \rangle E_{13}) \\ &= \nu(g)^{-1}( e\otimes \widetilde{g} \omega_0 + \langle \omega_0, z \rangle E_{13}).\end{align*}
Here, recall that $\widetilde{g} = \nu(g) g^{-1}$.  Thus we have 
\[|| g^{-1} n^{-1} e\otimes \omega_0||^2 = |\nu(g)|^{-2}\left( |\langle \widetilde{g} \omega_0, r_0(i)\rangle|^2 + |\langle \omega_0, z \rangle|^2\right)\]
using that $\omega_0$ is rank one.  Additionally, we have
\begin{align*} pr(g^{-1} n^{-1} e\otimes \omega_0) &= -\frac{1}{4} \nu(g)^{-1} \left( \langle \widetilde{g} \omega_0, r_0(-i)\rangle x^2 - 2i \langle \omega_0, z \rangle xy + \langle \widetilde{g} \omega_0, r_0(i)\rangle y^2\right) \\ &= -\frac{1}{4}\nu(g)^{-1}\left( \alpha^* x^2 - 2i \beta xy + \alpha y^2\right)\end{align*}
where $\alpha = \langle \omega_0, g r_0(i)\rangle$ and $\beta= \langle \omega_0, z\rangle.$  Thus we obtain
\begin{align*} f_1^\omega(g,s) &= \frac{|\nu(g)|^{s}}{4^n} \int_{\GL_1}\int_{z }{\psi^{-1}(\langle \omega, z\rangle) |t|^{s+n}\left( \alpha^* x^2 - 2i \beta xy + \alpha y^2\right)^n e^{-\pi t^2(|\alpha|^2 + \beta^2)}\,dz\,dt} \\&= \frac{|\nu(g)|^{s}}{4^n} \int_{\GL_1}\int_{\R }{\psi^{-1}(\lambda \beta) |t|^{s+n}\left( \alpha^* x^2 - 2i \beta xy + \alpha y^2\right)^n e^{-\pi t^2(|\alpha|^2 + \beta^2)}\,d\beta\,dt}\end{align*}
where $\lambda = a$ is the constant satisfying $\langle \omega, z\rangle = \lambda \langle \omega_0, z \rangle = \lambda \beta$.

Hence
\begin{align*} f_1^\omega(g,s) &= \frac{|\nu(g)|^{s}}{4^n}\sum_{k}\binom{n}{k} (xy)^{k}(\alpha^* x^2 + \alpha y^2)^{n-k} (-2i)^{k}\int_{\GL_1}|t|^{s+n} e^{-\pi t^2 |\alpha|^2} \\ &\;\;\; \times \left(\int_{\R}{\beta^k e^{-\pi t^2 \beta^2} e^{-2\pi i \lambda \beta}\,d\beta}\right)\,dt.\end{align*}
Because 
\[t^2 \beta^2 + 2i \lambda \beta = (|t|\beta+ i \lambda/|t|)^2 + \lambda^2/|t|^2,\]
this inner integral is
\[e^{-\pi \lambda^2/|t|^2} \int_{\R}{\beta^k e^{-\pi (|t|\beta+ i \lambda/|t|)^2}\,d\beta}.\]
The $k=0$ term then gives
\begin{align*} \text{$k=0$ term} &= \frac{|\nu(g)|^{s}}{4^n} (\alpha^* x^2 + \alpha y^2)^{n} \int_{\GL_1}{|t|^{s+n-1} e^{-\pi (t^2 |\alpha|^2+\lambda^2/t^2)}\,dt} \\ &= 2 \frac{|\nu(g)|^{s}}{4^n} (\alpha^* x^2 + \alpha y^2)^{n} \left(\frac{|\lambda|}{|\alpha|}\right)^{(s+n-1)/2}K_{(s+n-1)/2}(2\pi |\lambda||\alpha|). \end{align*}
Thus, at $s=n+1$, the coefficient of $x^{2n}$ is
\begin{align*} 2 \frac{|\nu(g)|^{n+1}}{4^n}|\lambda|^{n} \left(\frac{\alpha^*}{|\alpha|}\right)^{n}K_n(2\pi |\lambda||\alpha|) &= 2 \frac{|\nu(g)|^{n+1}}{4^n} |\lambda|^{n} \left(\frac{|\langle 2\pi \omega, g r_0(i)\rangle|}{\langle 2\pi \omega, g r_0(i)\rangle}\right)^{n} K_n(|\langle 2\pi \omega, g r_0(i)\rangle|)\\ &= 2 \frac{|\lambda|^n}{4^n} \Wh_{2\pi \omega}^{n}(g).\end{align*}
Here, recall that $|\lambda||\alpha| = |\langle \omega, g r_0(i)\rangle|$, and we have used that $n$ is even so that we don't have to worry about the sign of $\lambda$.

The coefficient of $x^{n+v}y^{n-v}$ with $v > 0$ comes from various of the terms with $k >0$, and gives a complicated mess of $K$-Bessel functions.  However, by the main result of \cite{pollackQDS}, these $K$-Bessel functions must ultimately combine and simplify, using the various identities among Bessel functions, to a single $K_v(\bullet)$.

\begin{remark} Note that we are using the normalization
\[K_s(y) = \frac{1}{2}\int_{0}^{\infty}{t^{s} e^{-y(t+t^{-1})/2}\,\frac{dt}{t}}.\]
\end{remark}

\subsubsection{Combining finite and archimedean} Combining the above computations of $f_1^\omega(g,s)$ at the finite and archimedean places, we have proved the following.
\begin{proposition} Suppose $g \in H_J(\R)$, $\omega = a e \otimes (0,0,0,1)$, and $\Phi_p$ restricted to $F E_{13} + F E_{23}$ is the characteristic function of $\mathcal{O} E_{13}+\mathcal{O} E_{23} = \Z_p E_{13} \oplus \Z_p E_{23}$ for every finite prime $p$.  Then
\[f_1^\omega(g,s=n+1) = \frac{2 (2n)!}{4^n} \sigma_n(|a|) \mathcal{W}_{2\pi \omega}(g).\]
\end{proposition}

\subsection{The degenerate Heisenberg Eisenstein series} Putting everything together, we have proved the following result.
\begin{corollary}\label{cor:eisConFC1} Suppose that $n > 0$ is even is such that $E(g,\Phi,s;n)$ is a modular form of weight $n$ at $s=n+1$.  Assume moreover that for all $p < \infty$, $\Phi_p$ is such that when restricted to $F E_{13} + F E_{23}$ it is the characteristic function of $\Z_p E_{13} \oplus \Z_p E_{23}$.  Denote by $E_{hol}(g,s,n)$ the Siegel Eisenstein series on $H_J$ defined as
\[E_{hol}(g,s,n) = \sum_{\gamma \in P_{Sieg}(\Q)\backslash H_J(\Q)}{f(\gamma g,n)}\]
with $f(p,s,n) = |\nu(p)||\lambda(p)|^{s}$ for $p \in P_{Sieg}(\A)$ and $f(gk,n) = j(k,i)^{-n}f(g,n)$ for $k \in K_H^1$.  Then for $g \in H_J(\R)$,
\begin{align*} E(n(x) g,\Phi,s=n+1)_0 &= \frac{\zeta(n)\Gamma(n)}{(4\pi)^{n}}\left(E_{hol}(g,n)x^{2n} + E_{hol}'(g,n)y^{2n}\right) \\  &\;\;\; + |\nu(g)|^{n+1}\frac{\zeta(n+1)(-1)^{n/2} \left(\frac{1}{2}\right)_n}{(2\pi)^n} x^n y^n + \sum_{\omega \in W_J(\Q)}{a(\omega) e^{2\pi i \langle \omega, x \rangle} \mathcal{W}_{2 \pi \omega}(g)}\end{align*}
for some coefficients $a(\omega)$.  If $n > \dim(J)+1$ and $\omega = a(0,0,0,1)$ with $a \in \Z$, then $a(\omega) = \frac{2 (2n)!}{4^n} \sigma_n(|a|)$.  \end{corollary}

\section{The minimal modular form}\label{sec:mmf} In this section we prove Theorem \ref{thm:thetaAgain}.  In the first subsection, we prove Proposition \ref{prop:EisReg1}.  In the second subsection, we put together all the pieces to complete the proof of Theorem \ref{thm:thetaAgain}.

\subsection{The value of the Eisenstein series at its special point} As mentioned, the purpose of this section is to spell out a proof of Proposition \ref{prop:EisReg1}, which is essentially contained in \cite{ganATM,ganSW,ganSavin,grossWallach2}. Crucial to our arguments is the following proposition.

\begin{proposition}\label{prop:Mws} Denote by $w_0$ the element of the Weyl group that takes $N$ to its opposite, and by $M(w_0,s)$ the intertwiner which for $Re(s) >> 0$ is given by the absolutely convergent integral
\[M(w_0,s)f(g,s;4) = \int_{N(\R)}{f(w_0^{-1}ng,s;4)\,dn}.\]
Here we are working on $G(\R) = E_{8,4}(\R)$ and $f$ is our special inducing section defined above.  Then $M(w_0,s)f(g,s;4) = h(s) f(g,29-s;4)$ where $h(s)$ is a meromorphic function of $s$ which is regular at $s=5$ and vanishes there.\end{proposition}

We give the proof of the proposition below.  It is via factorization of the intertwining operator and reduction to rational rank one. Denote by $P_0$ the minimal standard parabolic on $G=E_8$, so that $P_0$ defines a root system of type $F_4$.  We first prove the following corollary of Proposition \ref{prop:Mws}.
\begin{corollary}\label{cor:ThetaMF} The Eisenstein series $E(g,s;4)$ is regular at $s=5$ and defines a modular form of weight $4$ at this point.  Moreover, its constant term along $P_0$ consists of just two terms: $f(g,s=5;4)$ and $f_1^0(g,s=5;4)$.\end{corollary}
\begin{proof} First, from the functional equation of Eisenstein series, one obtains 
\begin{equation}\label{eqn:EisFE} E(g,s;4) = c_f(w_0,s) h(s) E(g,29-s;4).\end{equation}
The Eisenstein series $E(g,s;4)$ is spherical at \emph{every} finite place (because the octonions are split at every finite place) and thus the finite part $c_f(w_0;s)$ of the $c$-function can be computed.  In fact, one gets (e.g. \cite{ganATM,ginzburgRallisSoudry})
\[c_f(w_0;s) = \frac{\zeta(2s-29)\zeta(s-28)\zeta(s-23)\zeta(s-19)}{\zeta(2s-28)\zeta(s)\zeta(s-5)\zeta(s-9)}.\]

It is clear that the left-hand side of (\ref{eqn:EisFE}) is nonzero at $s=5$, by examining the contribution of $f(g,s;4)$ to its constant term.  Because $h(5) = 0$, it follows that $E(g,s;4)$ has pole at $s=24$.  Because, as proved in \cite{ganATM}, the order of the pole is at most one at $s=24$, one concludes that $E(g,s;4)$ is regular at $s=5$. More precisely, $E(g,5;4)$ is a nonzero constant times $Res_{s=24}E(g,s;4)$.

Now, as explained in \cite{ganATM}, at most two terms contribute to the constant term along $P_0$ of the residue $Res_{s=24}E(g,s;4)$. It follows that all but two terms in the constant term of $E(g,s;4)$ along $P_0$ vanish at $s=5$.  Since we have already seen two of these terms above, $f(g,s;n)$ and $f_1^0(g,s;n)$, these are the only terms that contribute to the constant term of $E(g,s;4)$ at $s=5$.

It follows immediately from our calculation of $f(g,s=n+1;n)$ and $f_1^0(g,s=n+1;n)$ above and \cite[section 11]{pollackQDS} that $\mathcal{D}_4$ annihilates both $f(g,s=5;4)$ and $f_1^0(g,s=5;4)$.  Thus $\mathcal{D}_4$ annihilates the constant term of the $E(g,s=5;4)$. Because applying $\mathcal{D}_4$ commutes with taking the constant term, $\mathcal{D}_4 E(g,s;4)$ has constant term $0$.  But an easy application of \cite[Theorem 7.3.1]{pollackQDS} yields $\mathcal{D}_4 E(g,s;4) = (s-5) E(g,f_s';4)$ for another $G(\widehat{\Z})$-spherical, $K$-finite flat section $f_s'$.  Thus the constant term of $E(g,f_s';4)$ is regular at $s=5$, and so $E(g,f_s';4)$ is regular, and thus $\mathcal{D}_4 E(g,s=5;4) = 0$.  That is, $E(g,s=5;4)$ is a modular form on $G$ of weight $4$.  This completes the proof of the corollary.\end{proof}

\subsubsection{Archimedean intertwiner} It remains to explain the proof of Proposition \ref{prop:Mws}.

Denote by $\alpha_i$ the simple roots of our $F_4$-root system, so that they are labeled (in order) $1,2,3,4$ when the root diagram is 
\[\circ---\circ==>==\circ---\circ.\]
That is, $1,2$ label the long simple roots and $3,4$ label the short simple roots.  Denote by $w_i$ the simple reflection in the root $\alpha_i$ and $[i_1,i_2,\ldots,i_k]$ for the composition $w_{i_1} w_{i_2} \cdots w_{i_k}$.  The long intertwiner \cite{ganATM} in this notation is $w_0=[1,2,3,2,1,4,3,2,1,3,2,4,3,2,1]$.

To compute $M(w_0,s)$ at the archimedean place, we factorize into the intertwiners for the simple roots $\alpha_i$ and use the cocycle property.  The short roots, fortunately, give spherical intertwiners on groups isogenous to $\SO(9,1)$.  Normalize the inner product $(\cdot, \cdot)$ on the $F_4$ roots spaces so that the long roots have norm squared equal to $2$.  With this normalization, if $\mu$ is a character of $P_0$, the result is that the $c$-function is
\[c(w_{k},\mu) = (\text{nonzero constant}) \frac{\Gamma((\mu, \alpha_{k}))}{\Gamma((\mu,\alpha_{k})+4)}\]
for $k = 3, 4$ corresponding to the short roots.

For the long roots, the intertwiners are no longer spherical, but we must only make a $\GL_2$-computation.  To do this, first denote by $V_{+}$ the three-dimensional subspace of $\Vm_4$ spanned by 
\[b_2^2:=x^8 + y^8, b_2^1 =: = x^2y^2(x^4 + y^4), b_2^0 :=x^4 y^4.\]
As in \cite[section 13]{ganSW}, 
define $f_1 = \frac{x+y}{2}, f_2 = \frac{x-y}{2}$.  Then $V_{+}$ is also the span of 
\[b_1^2:=f_1^8+f_2^8, b_1^1:=f_1^2f_2^2(f_1^4+f_2^4), b_1^0:=f_1^4f_2^4.\]
When we compute $M(w_2,\mu)$, it is convenient to use the first basis $b_2^i$, and when we compute $M(w_1,\mu)$ it is convenient to use the basis $b_1^j$.

More precisely, denote by $[\mu,b_j^k]$ the $K$-equivariant inducing section for $Ind_{P_0}^{G}(\delta_{P_0}^{1/2} \mu)$ whose value at $g=1$ is $b_j^k$.  Then for $i \in \{1,2\}$, $j\in \{1,2\}$ and $k \in \{0,1,2\}$, one has
\[M(w_i,\mu) [\mu,b_j^k] = \frac{\zeta_{\R}(s)}{\zeta_{\R}(s +1)} \frac{\left(\frac{1-s}{2}\right)_k}{\left(\frac{1+s}{2}\right)_k} [w_i(\mu),b_j^k]\]
where $s= \langle \mu, \alpha_i^\vee\rangle$. Here $\zeta_{\R}(s) = \Gamma_{\R}(s) = \pi^{-s/2}\Gamma(s/2)$ and $(s)_k = s(s+1)\cdots (s+k-1)$.

To carry out the computation of $M(w_0,s)$ one then just puts together the above information.  The change of basis matrix between the $b_2^i$'s and the $b_1^i$'s is 
\[A = \left(\begin{array}{ccc} 2&2&1\\56&8&-4 \\ 140 &-20 & 6\end{array}\right)\]
as
\begin{align*} x^8+y^8 &= 2(f_1^8+f_2^8)+2(f_1^6f_2^2+f_1^2f_2^6)+f_1^4f_2^4\\ x^6y^2+x^2y^6 &= 56(f_1^8+f_2^8)+8(f_1^6f_2^2+f_1^2f_2^6)-4f_1^4f_2^4\\ x^4y^4 &= 140(f_1^8+f_2^8)-20(f_1^6f_2^2+f_1^2f_2^6)+6f_1^4f_2^4.\end{align*} 

Putting together the pieces, one gets the following proposition, which immediately implies Proposition \ref{prop:Mws}.
\begin{proposition} Denote by $\lambda_{s} = \delta_{P_0}^{-1/2} |\nu(s)|^{s}$ the normalized character defining $f(g,s;n)$.  Then up to exponential factors, the intertwiner 
\[M(w_0)[\lambda_s,b_2^0] = Z(s) A(s) [\lambda_{29-s},b_2^0]\]
with
\[Z(s) = \left(\frac{\Gamma_{\R}(2s-29)\Gamma_{\R}(s-28)\Gamma_{\R}(s-19)\Gamma_{\R}(s-11)\Gamma_{\R}(s-2)}{\Gamma_{\R}(2s-28)\Gamma_{\R}(s-26)\Gamma_{\R}(s-17)\Gamma_{\R}(s-9)\Gamma_{\R}(s)}\right) \left(\frac{\Gamma_{\C}(s-23)\Gamma_{\C}(s-14)}{\Gamma_{\C}(s-11)\Gamma_{\C}(s-2)}\right)\]
and
\[A(s) = \frac{(s-31)(s-29)(s-22)(s-20)(s-14)(s-12)(s-5)(s-3)}{(s-26)(s-24)(s-17)(s-15)(s-9)(s-7)s(s+2)}.\]
\end{proposition}

\subsection{Proof of Theorem \ref{thm:thetaE8}} At this point, we have shown that $E_J(g,s;4)$ is regular at $s=5$ and defines a modular form there.  By the fact that the local representation $\pi_p \subseteq Ind_{P_J(\Q_p)}^{G_J(\Q_p)}(|\nu|^{5})$ generated by the spherical vector is minimal \cite{ganSavin,magaardSavin}, $E_J(g,s=5;4)$ only has a constant term and rank one Fourier coefficients; all of its rank two, three and four Fourier coefficients are $0$.  Moreover, our computations above show the following. 

Denote by $P_{Sieg}$ the Siegel parabolic subgroup of $H_J$ which by definition is the stabilizer of the line spanned by $(0,0,0,1)$ in $W_J$.  Let $E_{hol}(g,s;n)$ be the Siegel Eisenstein series on $H_J$ defined as
\[E_{hol}(g,s;n) = \sum_{\gamma \in P_{Sieg}(\Q)\backslash H_J(\Q)}{f(\gamma g,n)}\]
with $f(p,s;n) = |\nu(p)||\lambda(p)|^s$ for $p \in P_{Sieg}(\A)$ and $f(gk,s;n) = j(k,i)^{-n}f(g,s;n)$ for $k \in K_H^1$.  The Eisenstein series $E_{hol}(g,s;4)$ is regular at $s=4$.  The value $E_{hol}(g,s=4;4) = 240 |\nu(g)|^{5} \Phi_{Kim}(g)$, i.e., it corresponds to the holomorphic modular form that is the multiple of $H_{Kim}$ with constant term $1$.  Indeed, this is the result of \cite{kimE7}.

Thus for $g \in H_J(\R)$ and $x \in (N/N_0)(\R)\simeq W_J(\R)$
\begin{align*} E(x g,s=5,\Phi;4)_0 &= \frac{\zeta(4)\Gamma(4)}{(4\pi)^{4}}\left(E_{hol}(g,s=4;4)x^{8} + E_{hol}'(g,s=4;4)y^{8}\right) \\ &\;\;\; + |\nu(g)|^{5}\frac{\zeta(5) \left(\frac{1}{2}\right)_4}{(2\pi)^4} x^4 y^4 + \sum_{\omega \in W_J(\Q)}{a(\omega) e^{2\pi i \langle \omega, x \rangle}\mathcal{W}_{2 \pi \omega}(g)}\end{align*}
for some coefficients $a(\omega)$.  

\subsubsection{The nonconstant terms} To finish the proof, we must analyze the nonconstant terms.  We do this by applying Gan's Siegel-Weil theorem \cite{ganSW} for $G_2 \times F_4^{an} \subseteq G_J=E_{8,4}$.  Here recall that $F_4^{an}$ is the anisotropic $F_4$ defined to be the fixator of $1_J$ in the exceptional cubic norm structure $J = H_3(\Theta)$.  We only require the following much weaker form of it.

\begin{theorem}[Gan] \label{thm:EisRegSW} Denote by $\theta(\mathbf{1})(g)$ the theta-lift of the constant function $1$ on $F_4^{an}$ to $G_2$, i.e.,
\[\theta(\mathbf{1})(g) = \int_{[F_4^{an}]}{E((g,h),5;4)\,dh}.\]
Then the difference $E^{G_2}(g,5;4) - \theta(\mathbf{1})(g)$ is a weight $4$, level one \emph{cuspidal} modular form on $G_2$. \end{theorem}
Gan's Siegel-Weil theorem proves that the above difference is $0$, and moreover that it is $0$ for a large family of inducing sections.  However, we only require that the difference is cuspidal, and only for this one particular section.  These simplifications make the necessary result easier to prove, which is why we state it in this weaker form.  

Now, to finish the proof of Theorem \ref{thm:thetaAgain}, we must evaluate the rank one Fourier coefficients of $E(g,s=5;4)$.  Because this modular form is spherical, it suffices to evaluate $a_{\theta}(a(0,0,0,1))$ for positive integers $a$.  To do this, we will use Theorem \ref{thm:EisRegSW}, together with the following lemma.

Recall the definitions of $\Omega_I(\omega_0)$ and $\Omega_E(\omega_0)$ from subsection \ref{subsec:integralG2}.
\begin{lemma}\label{lem:a1sEq} Suppose $\omega_0 \in W_{F}$ is rank one.  Then $\Omega_{I}(\omega_0)$ and $\Omega_{E}(\omega_0)$ are each singletons, consisting of the elements $(a,b I, c I^\#,d)$ and $(a, b E, c E^\#, d)$, respectively.\end{lemma}
\begin{proof} By equivariance and scaling, we may suppose that $\omega_0 = (1,0,0,0)$.  But then, 
\[\Omega_{E}((1,0,0,0)) = \{(1,X,X^\#,N(X)): (X,E^\#) = (X^\#,E) = N(X) = 0\}.\]
However, the only such $X$ is $0$, and similarly with $I$ in place of $E$.  This proves the lemma.\end{proof}

Because modular forms that are cusp forms only have rank four coefficients, it follows from Lemma \ref{lem:a1sEq} that the rank one Fourier coefficients of $E_J(g,s=5;4)$ are equal to the rank one Fourier coefficients of the similar Eisenstein series on $G_2$.  But these coefficients were computed in Corollary \ref{cor:eisConFC1}.  This completes the proof of Theorem \ref{thm:thetaAgain}.

\bibliography{QDS_Bib} 

\providecommand{\bysame}{\leavevmode\hbox to3em{\hrulefill}\thinspace}
\providecommand{\MR}{\relax\ifhmode\unskip\space\fi MR }
\providecommand{\MRhref}[2]{%
  \href{http://www.ams.org/mathscinet-getitem?mr=#1}{#2}
}
\providecommand{\href}[2]{#2}
\begin{thebibliography}{{Pol}18a}

\bibitem[EG96]{elkiesGrossIMRN}
Noam~D. Elkies and Benedict~H. Gross, \emph{The exceptional cone and the
  {L}eech lattice}, Internat. Math. Res. Notices (1996), no.~14, 665--698.
  \MR{1411589}

\bibitem[Gan00a]{ganATM}
Wee~Teck Gan, \emph{An automorphic theta module for quaternionic exceptional
  groups}, Canad. J. Math. \textbf{52} (2000), no.~4, 737--756. \MR{1767400}

\bibitem[Gan00b]{ganSW}
\bysame, \emph{A {S}iegel-{W}eil formula for exceptional groups}, J. Reine
  Angew. Math. \textbf{528} (2000), 149--181. \MR{1801660}

\bibitem[Gan11]{ganRegularized}
\bysame, \emph{A regularized {S}iegel-{W}eil formula for exceptional groups},
  Arithmetic geometry and automorphic forms, Adv. Lect. Math. (ALM), vol.~19,
  Int. Press, Somerville, MA, 2011, pp.~155--182. \MR{2906908}

\bibitem[GG99]{ganGross}
Benedict~H. Gross and Wee~Teck Gan, \emph{Commutative subrings of certain
  non-associative rings}, Math. Ann. \textbf{314} (1999), no.~2, 265--283.
  \MR{1697445}

\bibitem[GGS02]{ganGrossSavin}
Wee~Teck Gan, Benedict Gross, and Gordan Savin, \emph{Fourier coefficients of
  modular forms on {$G_2$}}, Duke Math. J. \textbf{115} (2002), no.~1,
  105--169. \MR{1932327}

\bibitem[GRS97]{ginzburgRallisSoudry}
David Ginzburg, Stephen Rallis, and David Soudry, \emph{On the automorphic
  theta representation for simply laced groups}, Israel J. Math. \textbf{100}
  (1997), 61--116. \MR{1469105}

\bibitem[GS05]{ganSavin}
Wee~Teck Gan and Gordan Savin, \emph{On minimal representations definitions and
  properties}, Represent. Theory \textbf{9} (2005), 46--93. \MR{2123125}

\bibitem[GW94]{grossWallach1}
Benedict~H. Gross and Nolan~R. Wallach, \emph{A distinguished family of unitary
  representations for the exceptional groups of real rank {$=4$}}, Lie theory
  and geometry, Progr. Math., vol. 123, Birkh\"auser Boston, Boston, MA, 1994,
  pp.~289--304. \MR{1327538}

\bibitem[GW96]{grossWallach2}
\bysame, \emph{On quaternionic discrete series representations, and their
  continuations}, J. Reine Angew. Math. \textbf{481} (1996), 73--123.
  \MR{1421947}

\bibitem[JR97]{jiangRallis}
Dihua Jiang and Stephen Rallis, \emph{Fourier coefficients of {E}isenstein
  series of the exceptional group of type {$G_2$}}, Pacific J. Math.
  \textbf{181} (1997), no.~2, 281--314. \MR{1486533}

\bibitem[Kim93]{kimE7}
Henry~H. Kim, \emph{Exceptional modular form of weight {$4$} on an exceptional
  domain contained in {${\bf C}^{27}$}}, Rev. Mat. Iberoamericana \textbf{9}
  (1993), no.~1, 139--200. \MR{1216126}

\bibitem[KP04]{kazhdanPolishchuk}
D.~Kazhdan and A.~Polishchuk, \emph{Minimal representations: spherical vectors
  and automorphic functionals}, Algebraic groups and arithmetic, Tata Inst.
  Fund. Res., Mumbai, 2004, pp.~127--198. \MR{2094111}

\bibitem[Lok00]{loke1}
Hung~Yean Loke, \emph{Restrictions of quaternionic representations}, J. Funct.
  Anal. \textbf{172} (2000), no.~2, 377--403. \MR{1753179}

\bibitem[Lok03]{loke2}
\bysame, \emph{Quaternionic representations of exceptional {L}ie groups},
  Pacific J. Math. \textbf{211} (2003), no.~2, 341--367. \MR{2015740}

\bibitem[MS97]{magaardSavin}
K.~Magaard and G.~Savin, \emph{Exceptional {$\Theta$}-correspondences. {I}},
  Compositio Math. \textbf{107} (1997), no.~1, 89--123. \MR{1457344}

\bibitem[{Pol}18a]{pollackQDS}
A.~{Pollack}, \emph{{The Fourier expansion of modular forms on quaternionic
  exceptional groups}}, ArXiv e-prints (2018).

\bibitem[Pol18b]{pollackLL}
Aaron Pollack, \emph{Lifting laws and arithmetic invariant theory}, Cambridge
  J. Math. (to appear) (2018).

\bibitem[Sie39]{siegel}
Carl~Ludwig Siegel, \emph{Einf\"{u}hrung in die {T}heorie der {M}odulfunktionen
  {$n$}-ten {G}rades}, Math. Ann. \textbf{116} (1939), 617--657. \MR{0001251}

\bibitem[Ta00]{thang}
Nguy\^e\~n~Qu\^o\'c Th\v~a\'ng, \emph{Number of connected components of groups
  of real points of adjoint groups}, Comm. Algebra \textbf{28} (2000), no.~3,
  1097--1110. \MR{1742643}

\bibitem[Wal03]{wallach}
Nolan~R. Wallach, \emph{Generalized {W}hittaker vectors for holomorphic and
  quaternionic representations}, Comment. Math. Helv. \textbf{78} (2003),
  no.~2, 266--307. \MR{1988198}

\end{thebibliography}
\bibliographystyle{amsalpha}
\end{document}